\documentclass[11pt,reqno,a4paper,]{article}
\usepackage{amssymb,amsmath}
\usepackage{booktabs}     
\usepackage{multirow}     
\usepackage{array}        
\usepackage{tikz}
\usepackage{amsthm}
\usepackage{epsfig}
\usepackage{graphicx}
\usepackage{subcaption}
\usepackage{listings}
\usepackage{amssymb,amsmath}
\usepackage{enumerate}
\usepackage{cite}
\usepackage{makecell}
\newtheorem{Theorem}{Theorem}

\newtheorem{Lemma}{Lemma}

\theoremstyle{definition}
\newtheorem{Definition}{Definition}
\numberwithin{Theorem}{section}
\numberwithin{Lemma}{section}
\newtheorem{Example}{Example}
\numberwithin{Corollary}{section}
\numberwithin{Example}{section}
\numberwithin{Remark}{section}
\numberwithin{Proposition}{section}
\numberwithin{Definition}{section}
\date{}

\begin{document}
	\title{Neural network model for mathematical programming problems  with complementary constraints}
	\author{Anurag Jayswal, Ajeet Kumar$^*$
		\\\textit{Department of Mathematics and Computing, Indian Institute of Technology}
		\\\textit{(Indian School of Mines), Dhanbad-826004, Jharkhand, India}}
	\maketitle 
	\footnotetext {\hspace{-.65cm}$^{*}$Corresponding author
	{\\\textit{E-mail}: anurag@iitism.ac.in, anurag$\_$jais123@yahoo.com, (A. Jayswal),  22dr0039@iitism.ac.in, ajeetjionee@gmail.com, (A. Kumar)}}
	\noindent
	\textbf{Abstract}:
	In this paper, we propose a a gradient-based neural network model to solve the   mathematical programming problems  with complementary constraints (MPCC). In order to facilitate tractable optimization, the problem MPCC  is  transformed via a  regularized approach into a relaxed nonlinear optimization problem NLP($\beta$). After that employing  the penalty function and neural network model an estimate of the optimal solution of the problem NLP($\beta$) is obtained.  On the basis of Lyapunov stability theory and LaSalle invariance principle, the equilibrium point of proposed neural network is theoretically proven to be asymptotically stable and   capable to generate  optimal solution of the problem MPCC.
	 Further, we demonstrate the performance and dynamic behavior of the proposed neural network through various illustrative examples and its effectiveness via theoretical and numerical experiments.
	\section{Introduction} In this paper, we consider the following  constrained optimization problem
	\begin{align*}
		\textbf{(MPCC)}\hspace{2.1em}	\min \quad \ & f({w}) \\
		\text{subject to}\quad \ &  g_j(w)\leq 0, \quad \forall j=\overline{1,m},\\
		& h_i(w) = 0,  \,  \quad \forall i=\overline{ 1, l}, \\
		& 0 \leq G_k(w)  \perp  H_k(w) \geq 0, \quad  \forall k = \overline{1, s}.
	\end{align*}
	where all the functions $ f, g_j, h_i, H_k, G_k : \mathbb{R}^n \to \mathbb{R} $ are assumed to be  continuously differentiable. MPCC  is known as mathematical programming problems  with complementary constraints and the notation $ 0 \leq G_k(w)  \perp  H_k(w) \geq 0 , ~  \forall k = \overline{1, s}$ equivalently can be written as $G_k(w) \geq 0, ~H_k(w) \geq 0,~G_k(w)^T H_k(w) = 0, ~ \forall k = \overline{1, s}.$  Let $\mathcal{F}:=\{ w\in \mathbb{R}^n ~|~ g_j(w)\leq 0,  \forall j=\overline{1,m},~  h_i(w) = 0,  \forall i=\overline{ 1, l,}~ 0 \leq G_k(w) \ \perp \ H_k(w) \geq 0,\forall k = \overline{1, s}\}$ represents the set of all feasible points of the problem MPCC. The problem MPCC holds significant importance across various fields, including engineering design, economic equilibrium analysis, and multilevel game theory, see; Luo \textit{et al.} \cite{29}, and it has stimulated extensive research attention in the recent years.
\par 	The main difficulty in solving the problem MPCC, whether theoretically or numerically, stems from the presence of complementarity constraints. Such constraints cause the violation of standard constraint qualifications (such as, Mangasarian–Fromovitz constraints qualification (MFCQ),  Linear Independence Constraint Qualification (LICQ), etc.,) applicable to the traditional nonlinear optimization problems,  for instance, see Leyffer \cite{30}. As a result, the convergence conditions of standard optimization techniques are generally not met. To address this issue, several researchers over the past decade have proposed solution algorithms that exploit  the specific structure of the problem MPCC and eliminate the complications caused by the complementarity constraints. The most commonly adopted strategies to solve the problem MPCC includes penalty,  lifting, smoothing, regularization,   implicit programming, etc.,  are used to address inherent challenges associated with the problem MPCC (see \cite{1,31,32,33,34,35,15,36,37,41,42} and references cited therein). 	
 \par In the past few decades, many real-world problems have increasingly required real-time solutions; however, the traditional methods as mentioned above often struggle to provide them because of the complexity and high dimensionality of the problems. Due to intrinsic parallel processing capabilities and compatibility with hardware-based execution, neural networks  have emerged as powerful computational tools for providing real-time solutions and addressing the complex optimization problems. The use of neural networks in optimization began with Tank and
	Hopfield \cite{13} for the linear programming problem. Zhang \cite{5} considered a  Lagrange multiplier based neural networks  for general nonlinear programming problem and analyzed its stability. Later on, Kennedy and Chua \cite{14} explores the stability properties for a class of nonlinear programming circuit
	model and examine a neural network implementation. Subsequently, neural networks have been employed to address a wide range of optimization problems, including linear and nonlinear programming, variational inequalities, and complementarity problems, see \cite{16,17,18,19,20,21,22,23,24,25,26} and the references cited therein.
\par Inspired by the  adaptability and efficiency  of the neural networks, this study proposes a gradient-based neural network framework to address the MPCC problem. Due to  strongly convergent performance, we use the regularized approach considered by Kanzow and Schwartz \cite{1}.
The main outcomes of the study are as follows:
	\begin{itemize}
		\item[\textbf{(\textit{i})}] \textbf{Gradient-based neural network framework:} A new gradient-based neural network model is proposed to efficiently solve the problem MPCC.
		\item [\textbf{(\textit{ii})}] \textbf{Reformulation of the problem:} A  regularization approach is employed to transform the  problem MPCC into a continuously differentiable and constrained nonlinear programs NLP($\beta$) that typically do not exhibit the critical kinks in their feasible sets, as found in the  problem MPCC.
		\item [\textbf{(\textit{iii})}] \textbf{Method and Complexity:} By employing a penalty function technique,  the proposed neural network model has   simple one-layer architecture containing  $n$-neurons with low computational complexity in compare to the existing one Ezazipour and Golbabai \cite{28} model which has $n+2m$-neurons.
		\item [\textbf{(\textit{iv})}] \textbf{Theoretical and experimental validation:} Theoretical analysis proves that the  equilibrium point of the proposed neural network model  is asymptotically stable and  able to find the optimal solution of the problem MPCC, which is further verified through computer simulations that demonstrate the accuracy and efficiency of the model.
	\end{itemize}
	\par The rest structure of the paper is outlined as follows. In Section 2, we present  some notation and preliminary concepts associated with the problem MPCC and its transformed problem NLP($\beta$) and some definitions that are essential for obtaining the subsequent results presented in the paper.  In Section 3,  we construct the neural network model which able to find the optimal solution of the problem MPCC. In Section 4, we  analyze the asymptotic stability of the equilibrium point of  neural network and its relationship with the solution of the problem  MPCC. Further, in section 5, we provide computer simulations to showcase the utility and potential impact of the proposed approach.  Finally, in Section 6, we summarize the paper and provide the future development.
	\noindent
	\section{ Notation and preliminaries}
	
	In this section, we provide some notation and fundamental concepts of mathematical programming with complementary constraints, together with the essential definitions that serve as a basis for the theoretical developments presented in the subsequent sections. Let $\mathbb{R}^n$ be the $n$-dimensional Euclidean space and $\mathbb{R}_+^n$  be the non-negative orthant comprising all vectors in $\mathbb{R}^n$ with non-negative components.
		\begin{Definition}\cite{10}
			The differentiable function $f : \mathcal{U}\subseteq\mathbb{R}^n \to \mathbb{R}$ is said to convex  at $w^{*}\in \mathcal{U},$ if  
			$$f(w)-f(w^{*}) \geq (w-w^{*})^T\nabla f(w^{*}),~\forall w\in \mathcal{U}.$$
		\end{Definition}
	\par  Now, we define the index set of all active constraints at $ w^* \in \mathcal{F}, $ for the problem MPCC as
	$$
	\mathcal{I}_g(w^*) = \left\{ j\in \{\overline{1,m}\}  \mid g_j(w^*) = 0 \right\}.
	$$
	\par Further, for any $w^{*} \in \mathcal{F},$ we define the following index sets
	\begin{align*}
		\mathcal{I}_{00}(w^*) &:= \{ k\in \{\overline{1, s}\} \mid G_k(w^*) = 0, \, H_k(w^*) = 0 \}, \\
		\mathcal{I}_{0+}(w^*) &:= \{ k\in \{\overline{1, s}\} \mid G_k(w^*) = 0, \, H_k(w^*) > 0 \}, \\
		\mathcal{I}_{+0}(w^*) &:= \{ k\in \{\overline{1, s}\} \mid G_k(w^*) > 0, \, H_k(w^*) = 0 \}.
	\end{align*}
	It is important to note that the first subscript indicates the sign of 
	$G_k$ at a given point $w^*$, while the second subscript refers to the sign of $H_k$ at the same point.
	\begin{Definition}\cite{2}\label{def:1}
		A point   $w^*\in \mathcal{F}$ is said to be
		\begin{enumerate}
			\item[(\textit{i)}] \textbf{Weakly stationary (W-stationary)}: if there exists multipliers $\lambda \in \mathbb{R}^m, \xi \in \mathbb{R}^l, \eta, \zeta \in \mathbb{R}^s$ such that
			$$
			\nabla f(w^*) + \sum_{j=1}^m \lambda_i \nabla g_i(w^*) + \sum_{i=1}^l \xi_i \nabla h_i(w^*) - \sum_{k=1}^s \eta_k \nabla G_k(w^*) - \sum_{k=1}^s \zeta_k \nabla H_k(w^*) = 0.$$
			and
			$$
			\lambda_j \geq 0, \quad \lambda_j g_j(w^*) = 0 \  (j=\overline{1,m}),
			$$
			$$
			\eta_k = 0 \ (k \in \mathcal{I}_{+0}(w^*)), \quad \zeta_k = 0 \ (k \in \mathcal{I}_{0+}(w^*)),
			$$
			
			\item[(\textit{ii)}] \textbf{Clarke-stationary (C-stationary)}: if $w^*$ is a W-stationary and $\gamma_k \nu_k \geq 0$ for all $k \in \mathcal{I}_{00}(w^*)$,
			
			\item[(\textit{iii)}]\textbf{Mordukhovich stationary( M-stationary)}: if $w^*$ is a W-stationary and either $\eta_k > 0, \zeta_k > 0$, or $\eta_k \zeta_k = 0$ for all $k \in \mathcal{I}_{00}(w^*),$
			
			\item[(\textit{iv)}] \textbf{Strongly stationary (S-stationary)}: if $w^*$ is a W-stationary and $\eta_k \geq 0, \ \zeta_k \geq 0$ for all $k \in \mathcal{I}_{00}(w^*)$.
		\end{enumerate}
	\end{Definition}
Graphical view of these stationary conditions are depicted in the  Figure \ref{fig:1}.	The difference between  these  stationarity conditions comes from the multiplier index set $\mathcal{I}_{00}(w^*).$ If $\mathcal{I}_{00}(w^*)$ is empty, then all  stationarity conditions are coincide. The relationships between these stationarity conditions are organized as follows:
	 $$\text{S-stationarity} \Rightarrow \text{M-stationarity}\Rightarrow \text{C-stationarit} \Rightarrow \text{W-stationarity.}$$	
	
	
	
	
	
	
		\begin{figure}
		\begin{tabular}{cc}
			
		\begin{tikzpicture}[scale=1.2]
		\node at (0,-2.6) {(a) W-stationary};
		\node[black] at (0,-2.2) {$(\eta_k, \zeta_k~ \text{free})$};
		
		\fill[green!40] (0,0) rectangle (-1.5,1.5);
		\fill[green!40] (-1.5,-1.5) rectangle (0,0);
		\fill[green!40] (-1.5,-1.5) rectangle (0,0);
		\fill[green!40] (0,0) rectangle (1.5,1.5);
		\fill[green!40] (0,0) rectangle (1.5,-1.5);
		\draw[line width=4pt, green!100] (-1.5,0) -- (1.5,0);
		\draw[line width=4pt, green!100] (0,-1.5) -- (0,1.5);
		
		\draw[->] (1.5,0) -- (2,0) node[right] {$\eta_k$};
		\draw (-2,0) -- (-1.5,0);
		
		\draw[->] (0,1.5) -- (0,2) node[above] {$\zeta_k$};
		\draw (0,-2) -- (0,-1.5);
	\end{tikzpicture}
	\hspace{5.5em}
			&
		\begin{tikzpicture}[scale=1.2]
			\node at (0,-2.6) {(b) C-stationary};
			\node[black] at (0,-2.2) {$(\eta_k\zeta_k\geq 0)$};
			
			\fill[green!40] (-1.5,-1.5) rectangle (0,0);
			\fill[green!40] (-1.5,-1.5) rectangle (0,0);
			\fill[green!40] (0,0) rectangle (1.5,1.5);
			\draw[line width=4pt, green!100] (-1.5,0) -- (1.5,0);
			\draw[line width=4pt, green!100] (0,-1.5) -- (0,1.5);
			
			\draw[->] (1.5,0) -- (2,0) node[right] {$\eta_k$};
			\draw (-2,0) -- (-1.5,0);
			
			\draw[->] (0,1.5) -- (0,2) node[above] {$\zeta_k$};
			\draw (0,-2) -- (0,-1.5);
		\end{tikzpicture}
			
			\\
			
		\begin{tikzpicture}[scale=1.2]
			\node at (0,-2.6) {(c) M-stationary};
			\node[black] at (0,-2.2) {$(\eta_k,\zeta_k> 0~ \text{or}~\eta_k\zeta_k= 0)$};
			
			\fill[green!40] (0,0) rectangle (1.5,1.5);
			\draw[line width=4pt, green!100] (-1.5,0) -- (1.5,0);
			\draw[line width=4pt, green!100] (0,-1.5) -- (0,1.5);
			
			\draw[->] (1.5,0) -- (2,0) node[right] {$\eta_k$};
			\draw (-2,0) -- (-1.5,0);
			
			\draw[->] (0,1.5) -- (0,2) node[above] {$\zeta_k$};
			\draw (0,-2) -- (0,-1.5);
		\end{tikzpicture}
			\hspace{5.5em}
			&
			\begin{tikzpicture}[scale=1.2]
				\node at (0,-2.6) {(d) S-stationary};
				\node[black] at (0,-2.2) {$~~~~(\eta_k, \zeta_k\geq 0)$};
				
				\fill[green!40] (0,0) rectangle (1.5,1.5);
				\draw[line width=4pt, green!100] (0,0) -- (1.5,0);
				\draw[line width=4pt, green!100] (0,0) -- (0,1.5);
				
				\draw[->] (1.5,0) -- (2,0) node[right] {$\eta_k$};
				\draw (0,0) -- (-2.0,0);
				
				\draw[->] (0,1.5) -- (0,2) node[above] {$\zeta_k$};
				\draw (0,0) -- (0,-2.0);
			\end{tikzpicture}
	\end{tabular}
		\caption{Graphical view of the stationary conditions for an index $k\in\mathcal{I}_{00}(w^*).$} 
			\label{fig:1}
\end{figure}
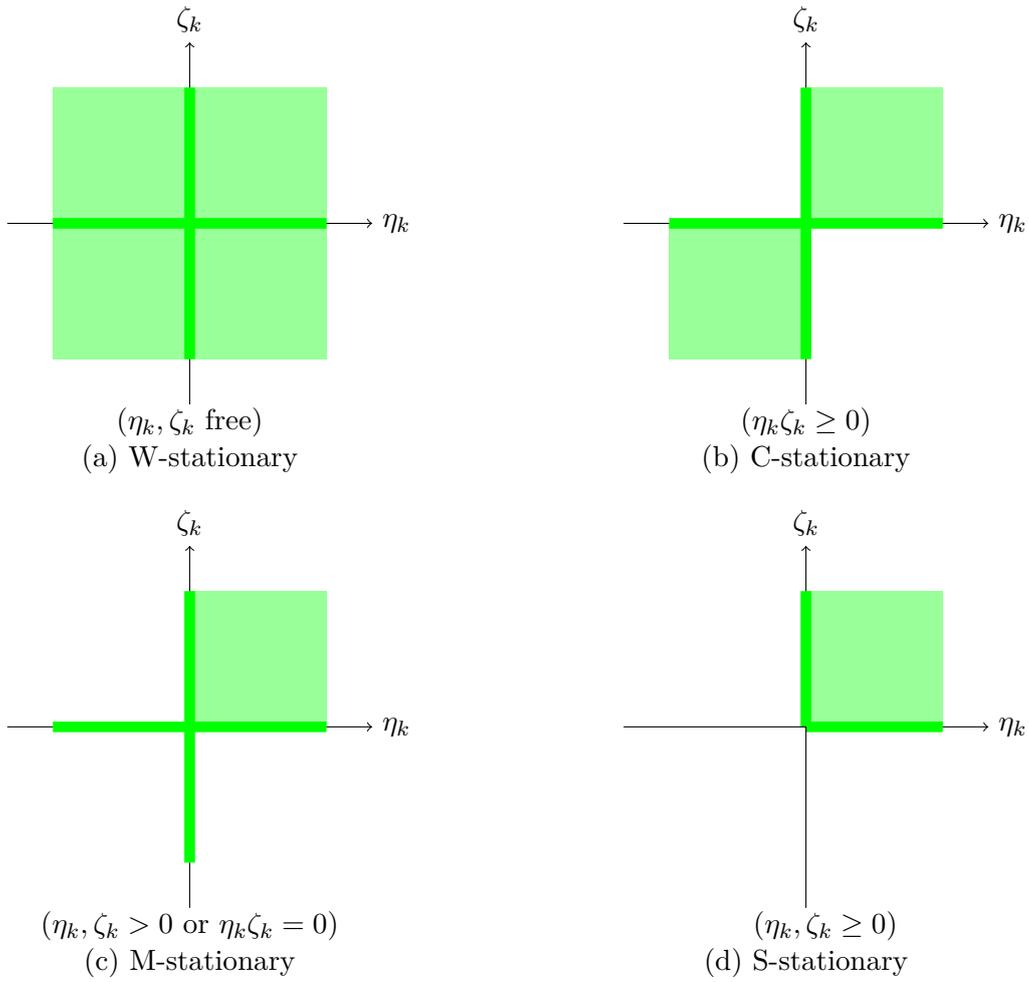
	\begin{Definition}\cite{1}
		A   point $w^* \in \mathcal{F}$  is said to satisfy the MPCC-linearly independent constraint qualification (MPCC-LICQ), if the gradients 
		\begin{align*}
			&\{\nabla g_j(w^*) \mid j \in \mathcal{I}_{g}(w^*)\} \cup \{\nabla h_i(w^*) \mid i=\overline{ 1, l}\} \\
			& ~~~~~~~\cup \{\nabla G_k(w^*) \mid k \in \mathcal{I}_{0+}(w^*)\cup \mathcal{I}_{00}(w^*)\} \cup \{\nabla H_k(x^*) \mid k \in \mathcal{I}_{+0}(w^*)\cup \mathcal{I}_{00}(w^*)\}
		\end{align*}
		are linearly independent.
	\end{Definition}

To solve the problem MPCC, whether theoretically or numerically, the main  difficulties stems from the presence of complementarity constraints, which violate the  standard constraints qualification (see \cite{30}), and generally lead to the feasible set  is nonsmooth and nonconvex, (see \cite{9}). To remove such critical kinks presented in the problem MPCC, 
 we use the following regularization approach considered by Kanzow and Schwartz  \cite{1}.
	\par Let $\phi(p,q)$ be a regularized function defined by
$$
	\phi(p,q) := 
	\begin{cases} 
		pq & if ~ p+q \geq 0, \\
		-\tfrac{1}{2}(p^2+q^2) & if ~ p+q < 0.
	\end{cases}
	$$
\begin{Lemma}\cite{1} The function $\phi$ exhibits the following properties:
	\begin{itemize}
		\item[(\textit{i})] $\phi$ is an \emph{NCP-function}, i.e., $\phi(p,q) = 0$ if and only if $p \geq 0$, $q \geq 0$, $pq = 0$.
		\item[(\textit{ii})] $\phi$ is continuously differentiable with gradient
	$$
		\nabla \phi(p,q) = 
		\begin{cases}
			\begin{pmatrix}
				q \\ p
			\end{pmatrix} & \text{if } p+q \geq 0,
			\\ 
			 
			 \begin{pmatrix}
				-p \\ -q
			\end{pmatrix} & \text{if } p+q < 0.
		\end{cases}
	$$
		\item[(\textit{iii})] $\phi$ has the property that
	$$
		\phi(p,q) \begin{cases}
			> 0 & \text{if } p> 0 \text{ and } q > 0, \\
			< 0 & \text{if } p < 0 \text{ or } q < 0.
		\end{cases}
		$$
	\end{itemize}
\end{Lemma}
Based on the $\phi$-function, we define the following continuously differentiable mapping 
 $\mathcal{B}: \mathbb{R}^n \to \mathbb{R}^s$ given by
\par $\mathcal{B}_k(w, \beta) := \phi \bigl(G_k(w) - \beta, \, H_k(w) - \beta \bigr)$\\
$
{}\hspace{5.4em}= \begin{cases}
	\bigl(G_k(w) - \beta \bigr)\bigl(H_k(w) - \beta \bigr), 
	& if ~ G_k(w) + H_k(w) \geq 2\beta, \\
	-\tfrac{1}{2} \Bigl( \bigl(G_k(w) - \beta \bigr)^2 + \bigl(H_k(w) - \beta \bigr)^2 \Bigr), 
	& if ~ G_k(w) + H_k(w) < 2\beta,
\end{cases}
$\\\\
where $\beta \geq 0$ is an arbitrary parameter. From a geometrical perspective, this leads to a set with the structure represented in the Figure \ref{fig:2}.
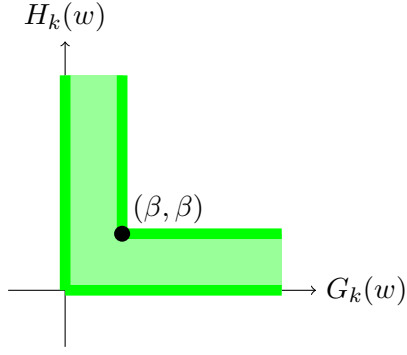
\begin{figure}
	\begin{center}
		\begin{tikzpicture}[scale=1.5]
			
			\draw[->] (-0.5,0) -- (2.2,0) node[right] {$G_k(w)$};
			\draw[->] (0,-0.5) -- (0,2.2) node[above] {$H_k(w)$};
			
			\fill[green!40] (0,0) rectangle (.5,1.9);
			\fill[green!40] (0,0) rectangle (1.9,.5);
			\draw[line width=4pt, green!100] (0,0) -- (1.9,0);
			\draw[line width=4pt, green!100] (0.5,0.5) -- (0.5,1.9);
			\draw[line width=4pt, green!100] (0.5,0.5) -- (1.9,0.5);
			\draw[line width=4pt, green!100] (0,0) -- (0,1.9);
			
			\fill (.5,.5) circle (2.0pt);
			\node[above right] at (.5,.5) {$(\beta,\beta)$};
			
		\end{tikzpicture}
	\end{center}
	\caption{Geometrical  view of the  $\mathcal{B}_k(w, \beta).$}
	\label{fig:2}
	
\end{figure}\\
	Now, using the function $\mathcal{B}_k(w, \beta),$ we get the following   regularized problem:
$$
	\begin{aligned}
	{}\hspace{5.1em}	\min \; & f(w) \\
		\text{subject to } 
		& g_j(w) \leq 0, \quad \forall j=\overline{1,m}, \\
		& h_i(w) = 0, \quad   \forall i=\overline{ 1, l}, \\
		& G_k(w) \geq 0, \quad \forall k = \overline{1, s}, \\
		& H_k(w) \geq 0, \quad \forall k = \overline{1, s}, \\
		& \mathcal{B}_k(w, \beta) \leq 0, \quad\forall k = \overline{1, s},
	\end{aligned}
	$$
Further,  for the simplification the above problem can be rewritten as
$$
\begin{aligned}
\textbf{(NLP($\beta$))}\hspace{5.1em}	\min \; & f(w) \\
	\text{subject to } 
	& h(w) = 0, \\
	& \mathcal{N}(w, \beta) \leq 0.
\end{aligned}
$$
\text{where} $ \mathcal{N}(w, \beta)=(g_j(w),- G_k(w), -H_k(w), ~\mathcal{B}_k(w, \beta))^T,~ h(w)=[h_i(w) ].$ Let $\mathcal{M}$ be the feasible set to the problem NLP($\beta$).
The optimization problem NLP($\beta$) is continuously differentiable and constrained nonlinear programs that typically do not exhibit the critical kinks as discussed above. Therefore, we expect that standard optimization software can be used to solve the NLP($\beta$) problems.
 
 
\par  The following theorem exhibits the relationship  between  the  problems NLP($\beta$) and  MPCC.
  \begin{Theorem}\cite{1}\label{thm:2.1}
  	Let $w^{\beta}$ be a KKT-point of the problem NLP($\beta$) with the  multipliers $(\xi^{\beta}, u^{\beta}),$ and  
   $(w^{\beta}, \xi^{\beta}, u^{\beta}) \to (w^*, \xi^*, u^*)$ as $\beta \to 0$. 
  If MPCC-LICQ holds at $w^*,$	then $w^*$ is an M-stationary point to the problem MPCC.
  \end{Theorem}
	\section{Construction  of neural network model}
	In general, to solve the constrained optimization problems using neural networks a key approach is to construct an appropriate energy function such that its minimum corresponds to the optimal solution of the considered problem. Therefore, to tackle the problem NLP($\beta$), we use the penalty function method to  transforms the constrained problem into an unconstrained one as described in Luenberger and Ye \cite{6}. 
	\begin{equation}
		\min \mathcal{E}(w, \beta)=f(w)+b \mathcal{P}
		\label{eq:1}
	\end{equation}
	where   $b$ is a positive constant  $\mathcal{P}:\mathbb{R}^n\to \mathbb{R}$ is a penalty function. 
	 Here, we assume that the penalty function $\mathcal{P}$ is continuous and $\mathcal{P}(w) \ge 0,  \forall w \in \mathbb{R}^n$ and $\mathcal{P}(w) = 0 \Leftrightarrow w \in \mathcal{M}$. Further, an effective penalty function for a set of inequality constraints $\{w : c_\alpha(w) \leq 0, \; \alpha = 1, 2, \ldots, \zeta\}
	$ is defined by  
	 $$
	\mathcal{P}(w) = \frac{1}{2}\|c^{+}(w)\|^{2} = \frac{1}{2}\sum_{\alpha=1}^{\zeta} \left( \max\{0, c_\alpha(w)\} \right)^{2}.
$$
	\par
And, thus we formulate $\mathcal{E}(w, \beta)$ (an energy function) to approximate the problem NLP($\beta$) as follows:
	\begin{equation}
		\mathcal{E}(w, \beta)=  f(w) + \frac{\lambda}{2} 
		\left\{ 
		\|  \mathcal{N}^+(w, \beta) \|^2 + \| h(w) \|^2
		\right\},
	\label{eq:2}
	\end{equation}
	where  $\lambda$   is a large penalty parameter. It can be seen for  large positive value of $\lambda$ and for every fixed $\beta,$ the energy function  $\mathcal{E}(w, \beta)$ 
 possesses both non-negativity and continuous differentiability.
 \par 
 In order to create neural network, we assume $w(.)$ is  some time-dependent variable and $w (t, w_0)$  represents a  trajectory starting from the initial point $w_0=w(t_0)$, where $t_0$ be an initial time.
Now, using the steepest descent method, we formulate the following neural network to efficiently handle the problem MPCC.
 \begin{equation}
 	\frac{dw}{dt} = -\nabla \mathcal{E}(w, \beta) = -\nabla f(w) - \lambda \left\{ \mathcal{N}^+(w,\beta) \nabla \mathcal{N}(w,\beta)  + h(w) \nabla h(w)  \right\}.\label{eq:3}
 \end{equation}
 Also, a simplified block diagram for the neural network \eqref{eq:3} is presented in the Figure \ref{fig:block}. 
  \begin{figure}[h!]
 	\centering
 	\includegraphics[width=0.8\textwidth]{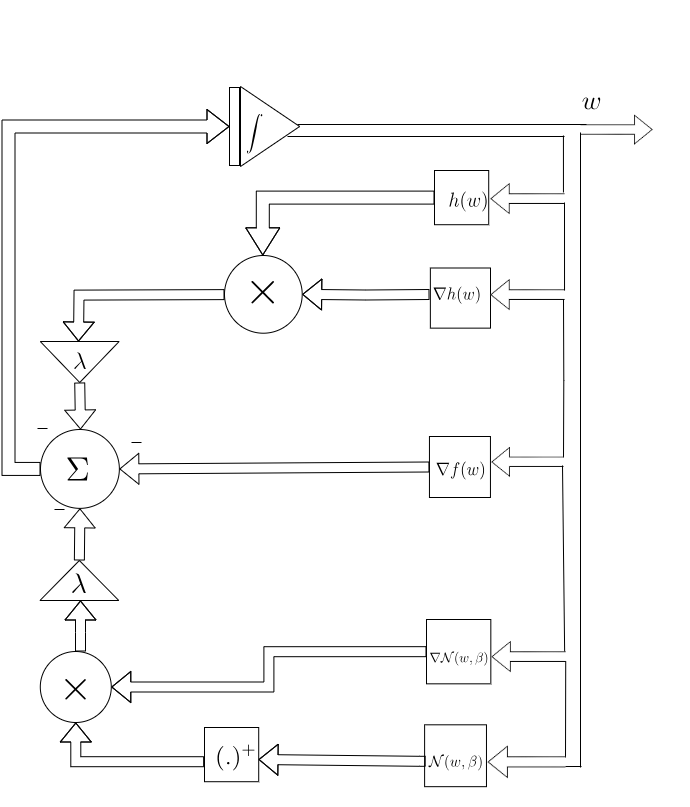}
 	\caption{The architectural design of the proposed neural network model \eqref{eq:3}.}
 	\label{fig:block}
 \end{figure}
 \par Now, on the lines of Luenberger and Ye \cite{6}, we give the following theorem which establish the relationship between the unconstrained problem \eqref{eq:1} and NLP($\beta$).
 \newpage
 \begin{Theorem}\label{thm:3.1}
 	Assume that $\{\beta^k\}$ is a sequence of nonzero terms approaching to zero as $k$ tends to infinity. 
 	Let $\{\lambda^k\}$ be a sequence of positive real numbers that increases with $k$ and diverges to infinity. 
 	If the sequence $\{w^k\}$ is generated using the penalty method. 
 	Then, any limit point of the sequence $\{w^k\}$  solve the   problem NLP($\beta$). 
 \end{Theorem}
 \begin{proof}
 	As the proof is analogous to the convergence theorem discussed in [Luenberger and Ye \cite{6}, pp. 404], hence, we omit the details.
 	\end{proof}

	\section{Stability and convergence analysis}
	In this section, we present  the stability and convergence properties of the neural network  (\ref{eq:3}) and  examine the relationship between the equilibrium of the neural network  along with the approximate solution of the problem  MPCC  as the  parameter $\beta $ approaches to zero. 
	
\begin{Theorem}
Let the point $w(t_0)=w_0$ be an intial point of the neural network \eqref{eq:3}.	If the level set $
	L(w_0) = \{ w \in \mathbb{R}^n \mid \mathcal{E}(w, \beta) \leq \mathcal{E}(w_0, \beta) \}$
	is bounded, then the neural network  \eqref{eq:3} possesses an equilibrium point 
	$w^*$, and there exists a strictly increasing sequence $\{t_k\}$ such that 
	\begin{itemize}
		\item [(\text{a})]$\lim\limits_{k \to +\infty} w(t_k, w_0) = w^*~\text{along with}~  \nabla \mathcal{E}(w^*, \beta)=0,$
		\item [(\text{b})] $\lim\limits_{t\to +\infty} \mathcal{E}[w(t), \beta]=\mathcal{E}(w^*, \beta).$
	\end{itemize}
\end{Theorem}

\begin{proof}
(\textit{a}) Let us define the trajectory set
	$
	\mathcal{T}(w_0) = \{\, w(t, w_0) : t \ge t_0 \,\} 
$ initiated from the initial point $w_0.$ Then the time derivative of the energy function $\mathcal{E}(w, \beta)$ along any trajectory $w(t, w_0) \in \mathcal{T}(w_0)$ is given by
	$$
\left.\frac{d\mathcal{E}(w, \beta)}{dt}\right|_{w(t, w_0)} = \nabla \mathcal{E}(w, \beta)^{T} \frac{dw}{dt},
	$$
which together with the neural network \eqref{eq:3}, yields
\begin{align}
\left.\frac{d\mathcal{E}(w, \beta)}{dt}\right|_{w(t, w_0)}  = -\| \nabla \mathcal{E}(w, \beta) \|^2 \le 0, \label{eq:4}
\end{align}
which shows that $\mathcal{E}(w, \beta)$ decreases monotonically \textit{i.e.} non-increasing along the trajectory $w(t, w_0)$.
	Also, since $\mathcal{E}(w, \beta)$ is continuous,  thus the trajectory remains within the level set
	$
 L(w_0) = \{ w \in \mathbb{R}^n \mid \mathcal{E}(w, \beta) \leq \mathcal{E}(w_0, \beta) \},
	$ and we have
	$$\mathcal{T}(w_0) \subseteq  L(w_0).$$
Further, since	the function $\mathcal{E}(w, \beta)$ is bounded below and non-increasing, thus the trajectory $w(t, w_0)$ is also bounded.
	Thus, for any strictly increasing sequence $\{\bar{t}_k \}$ with $ \bar{t}_k \to +\infty$, 
	the sequence $\{ w( \bar{t}_k, w_0) \}$ is bounded  and  admits a limit point ${w}^*$. 
	Therefore, there exists a subsequence $\{ t_k \} \subseteq \{\bar{t}_k\}$ such that
$$
	\lim\limits_{k \to +\infty} w(t_k, w_0) = w^*.
$$
Now, the inequality \eqref{eq:4} together with  the non-negativity of the energy function $\mathcal{E}(w, \beta),$ exhibits that  $\mathcal{E}(w, \beta)$  is  a Lyapunov function, which satify
	\begin{equation}
	\frac{d}{dt} \mathcal{E}(w, \beta) = 0 \Leftrightarrow \nabla \mathcal{E}(w, \beta) = 0. \label{eq:a}
\end{equation}
	And thus, according to LaSalle’s invariance principle \cite{8}, any trajectory of the neural network \eqref{eq:3} starting in $L(w_0)$ 
	approaches the largest invariant set $\mathcal{Q}$ such that
$$
	\mathcal{Q} \subseteq \Big\{\, w \in L(w_0) \mid \frac{d}{dt}\mathcal{E}(w, \beta) = 0 \,\Big\}	= \{\, w \in L(w_0) \mid \nabla \mathcal{E}(w, \beta) = 0 \,\}.~~~(by ~\eqref{eq:a})
	$$
	Therefore, $w^* \in \mathcal{Q}$ and $\nabla \mathcal{E}(w^*, \beta) = 0$.
	Hence, $w^*$ is an equilibrium point of the neural network \eqref{eq:3}.\\
	(\textit{b}) Since $t_k \to +\infty$, for every $k \geq 0,$ there exists an integer $K$ such that $t_k > t$ whenever $k \geq K$.  
	From part (a), we know that $ \mathcal{E}[w(t), \beta]$  decreases monotonically, and therefore
	\begin{center}
	  $\mathcal{E}[w(t_k), \beta] \leq \mathcal{E}[w(t), \beta].$
		\end{center}
	As $k \to +\infty$, we obtain
	\begin{center}
$\lim_{k \to +\infty}	\mathcal{E}[w(t_k), \beta] = \mathcal{E}[w^*, \beta]\leq \mathcal{E}[w(t), \beta].
	$
\end{center}
	Thus, for every $\varepsilon > 0$, there exists an integer $K_0,$ such that
	$$
	\mathcal{E}[w(t_k), \beta] < \mathcal{E}[w^*, \beta]+\varepsilon, \quad \forall k \geq K_0.
	$$
	Also, we have
$$	\mathcal{E}[w^*, \beta] = \inf_{t \ge 0} \{\mathcal{E}[w(t), \beta]\}.
	$$
Hence,	it follows from the above two  relations
	$$
	\lim_{t \to +\infty} \mathcal{E}[w(t), \beta] = \lim_{k \to +\infty} \mathcal{E}[w(t_k), \beta] =\mathcal{E}[w^*, \beta].
	$$
	This completes the proof.
\end{proof}
\begin{Theorem}
Let the level set $
\mathcal{C}(w_0) $ is bounded and $w(t, w_0)$ be the trajectory of  neural network \eqref{eq:3} starting from the initial point $w_0.$  Further, let $w^{\beta}$ be a KKT-point to the problem NLP($\beta)$ such that  $w^{\beta}\to w^*~\textit{as}~ \beta \to 0,$ where $w^*$ is an equilibrium point for the considered neural network \eqref{eq:3}.  
Also, assume that there exists a  neighborhood $\mathcal{V} \subseteq \mathbb{R}^n$ around $w^*$ such that the energy function $\mathcal{E}(w, \beta)$ is   convex, for all $w \in \mathcal{V}$ and  MPCC-LICQ holds at $w^*.$\\
	Under these assumptions, the following results are true:
	\begin{enumerate}
		\item  [(\text{a})] The trajectory $w(t, w_0)$  converges to an optimal solution of the penalized problem \eqref{eq:1}, and thus solve the considered constraints optimization problem MPCC.
		\item [(\text{b})]The neural network \eqref{eq:3} exhibits  an asymptotic stability over the equilibrium point  $w^*.$
	\end{enumerate}
\end{Theorem}

\begin{proof} (\textit{a}) 
	Let $\bar{w}$  be an optimal solution of  the function $\mathcal{E}(w, \beta)$  and $w^*$ be an equilibrium point of the  neural network \eqref{eq:3}. Now, consider the function $ \mathcal{L}: \mathbb{R}^n \to \mathbb{R}$ define by
	  $$
	  \mathcal{L}(w) = \frac{1}{2}\|w-\bar{w}\|^2.
	  $$
Then, the time derivative of $\mathcal{L}(w)$ along the trajectories of the neural network \eqref{eq:3}  is given by
	  \begin{equation}	\frac{d  \mathcal{L}(w)}{dt}=(w-\bar{w})^{T} \frac{dw}{dt} = -\,(w-\bar{w})^{T} \nabla \mathcal{E}(w, \beta). \label{eq:6}
	  \end{equation}
	  	Since, the energy function $\mathcal{E}(w, \beta)$ is continuous differentiable and  convex for all $w\in \mathcal{V}$, the equilibrium point $w^*$  of the neural network \eqref{eq:3} and optimal solution $\bar{w}$ of the function $\mathcal{E}(w, \beta)$  are coincide on $\mathcal{V}.$ \\
	 Again, using the convexity assumption of the function  $ \mathcal{E}(., \beta),$  we have 
	  \begin{equation*}
	  	\mathcal{E}(\bar{w}, \beta) -  \mathcal{E}(w, \beta)
	  	\;\geq\; (\bar{w}-w)^{T}\nabla  \mathcal{E}(w, \beta),  \forall ~ w \in \mathcal{V}, \label{eq:5}
	  \end{equation*}
	  which together with the relation \eqref{eq:6}, gives
	 \begin{equation}\frac{d  \mathcal{L}(w)}{dt} =- (w-\bar{w})^{T} \nabla \mathcal{E}(w, \beta)\leq	\mathcal{E}(\bar{w}, \beta) - \mathcal{E}(w, \beta)
	  \;\leq\; 0, \label{eq:b}
	\end{equation}
	  and thus we conclude that $\mathcal{L}(w)$ is a Lyapunov function of the neural network  \eqref{eq:3}.\\
	  Now, as per hypothesis, we define a bounded level set $\mathcal{C}(w_0)$ given by    
	  $$\mathcal{C}(w_0)=\{w\in \mathbb{R}^n~|~\mathcal{L}(w)\leq \mathcal{L}(w_0)\}.$$
	   	And thus, according to LaSalle’s invariance principle \cite{8}, any trajectory of the neural network \eqref{eq:3} starting in $\mathcal{C}(w_0)$ approaches the largest invariant set $\mathcal{S}$ such that
	  $$
	  \mathcal{S} \subseteq \{\, w \in \mathcal{C}(w_0) \;|\; \frac{d  \mathcal{L}(w)}{dt}  = 0 \,\}.
	  $$
	  Next, we characterize the invariant set $\mathcal{S}.$  So for
	  if, $w \in \mathcal{S} \Rightarrow\frac{d  \mathcal{L}(w)}{dt}  =	0  $, then from the inequality \eqref{eq:b}, we obtain
	  $$\mathcal{E}(\bar{w}, \beta) = \mathcal{E}(w, \beta),$$
	  which yields $w$ is an optimal solution of the function  $\mathcal{E}(w, \beta)$ and thus equilibrium point of the neural network \eqref{eq:3} and we have $\nabla\mathcal{E}(w, \beta)=0.$\\ 
	  Conversely, if $\nabla\mathcal{E}(w, \beta)=0 \Rightarrow \frac{dw}{dt} =0,$ and consequently from the equality \eqref{eq:6}, we have $\frac{d  \mathcal{L}(w)}{dt} =0,$ and hence $ \frac{d  \mathcal{L}(w)}{dt}=0 \Leftrightarrow \nabla\mathcal{E}(w, \beta)=0,$ and it follows that
	  $$
	  \mathcal{S}  \subseteq \{\, w \in \mathcal{C}(w_0) \;|\; \nabla\mathcal{E}(w, \beta)=0 \,\}.
	  $$	
	 Since $w^*$ is an equilibrium point of the neural network  \eqref{eq:3}, and thus will be the limiting point of  any trajectory  contained in the invariant set $\mathcal{S}.$ Therefore,  $w^* \in \mathcal{S},$ and thus an optimal solution of the function $\mathcal{E}(w, \beta).$\\
	 Next, in order to show that  $ \lim\limits_{t \to +\infty} w(t) =\bar{w}, $ we consider the following Lyapunov function 
	 $$
	 \mathcal{\hat{L}
	 }(w) = \frac{1}{2}\|w-w^*\|^2.
	 $$
	 Proceeding on the  similar way as of  $\mathcal{L}(w),$  one can easily verify that $\mathcal{\hat{L}}(w)$ is a non-increasing function along the trajectory of the neural network \eqref{eq:3}.
	 Therefore, for all $\epsilon>0,$ there exists $\hat{l}>0,$ such that
	 $$\mathcal{\hat{L}}(w(t))<\mathcal{\hat{L}}(w(t_{\hat{l}}))<\epsilon,~  \forall  t>t_{\hat{l}}, $$
	 $$  \Rightarrow  \frac{1}{2}\|w-w^*\|^2<\epsilon, ~  \forall  t>t_{\hat{l}}.$$
	It follows that
	  $$ \lim\limits_{t \to +\infty} w(t) =w^*=\bar{w}.$$
	  Hence, the trajectory $w(t, w_0)$  converges to an optimal solution $w^*(=\bar{w})$  of the penalized problem \eqref{eq:1}. Further, since  $w^{\beta}$ is a KKT-point to the problem NLP($\beta)$ such that  $w^{\beta}\to w^*~\textit{as}~ \beta \to 0$ and MPCC-LICQ holds at $w^*,$ it follows from Theorem \ref{thm:2.1}, that $w^*$ solves the problem MPCC.\\
	(\textit{b}) Let 
	$
	\mathcal{Z}(w_0) = \{\, w(t, w_0) : t \ge t_0 \,\} 
	$ be the set of trajectories starting from the initial point  $w_0$ of the neural network \eqref{eq:3}, and $\mathcal{E}(w, \beta)$ be an energy function along any trajectory $w(t, w_0) \in \mathcal{Z}(w_0)$. Proceeding on the similar lines of Theorem 4.1 (a), one has $\mathcal{E}(w, \beta)$ decreases monotonically. 	
	 Now, consider the function
	\begin{equation*}
		\mathcal{Y}(w) = \mathcal{E}(w, \beta) -\mathcal{E}(w^*, \beta), ~ \forall w \in \mathcal{V}.
	\end{equation*}
Since, the function $\mathcal{E}(w, \beta)$  decreases monotonically and $w^*$ is an optimal point of the function $\mathcal{E}(w, \beta),$ we have 
	$$\mathcal{Y}(w) > 0, ~\forall  w \neq w^*.$$ Also, $ \mathcal{Y}(w^*) = 0,$ and thus \( \mathcal{Y}(w) \) 
	is a positive definite function defined on $\mathcal{V}.$ Moreover, for any initial 
	point $ w_0,$ there exists a trajectory $w(t, w_0) $ of the neural network \eqref{eq:3} such that
	\begin{equation*}
		\left.\frac{d}{dt}\mathcal{Y}(w)\right|_{w(t, w_0)} = \nabla \mathcal{Y}(w)^{T} \frac{dw}{dt} 
		= -\|\nabla \mathcal{E}(w, \beta)\|^{2} < 0, ~ \forall w \neq w^*.
	\end{equation*}
		Since, the function $ \mathcal{Y}(w) $ is positive definite on $\mathcal{V}$ and its time derivative strictly negative for all  $w \neq w^*.$  Thus,
	 according to Lyapunov stability theorem \cite{40}, the neural network \eqref{eq:3} is asymptotically stable over the equilibrium point $w^*.$
This completes the proof.
\end{proof}
	\section{Computer Simulations}
In this section, we present  computer simulations for the proposed neural network. All the numerical experiments are carried out on a personal computer utilizing MATLAB R2025b. Further, the system of ordinary differential equations was integrated leveraging the ode45 solver, and random initial points are generated using MATLAB’s randi function. The results of the numerical experiments demonstrate the effectiveness and reliability of the proposed neural network model to address the MPCC problem.
	\begin{Example}\cite{38}\label{ex:1} Consider the  problem:
	$$
	\begin{aligned}
 \text{(MPCC1)}\quad & \min && w^T Qw \\
		&\text{subject to} && 0 \le w_1 \perp w_2 \ge 0,
	\end{aligned}
	$$
where		$
	Q = 
	\begin{bmatrix}
		1 & -1 \\
		-1 & 1
	\end{bmatrix}
	$
and $w = (w_1, w_2) \in \mathbb{R}^2.$ 
\par 
In order to  solve the problem MPCC1, we employ the proposed neural network \eqref{eq:3}. Using the above mentioned mathematical tools,  the numerical experiments with 10 random initial points,  shows that the solution trajectories $w_1(t)$ and $w_2(t)$ of the  neural network \eqref{eq:3}   converge  to the global minimum point  $(0,0)$ with the parameters $\lambda=10^{6}$ and $\beta=0.0001,$ as depicted  in the figures \eqref{fig:Fig4a} and \eqref{fig:Fig4b}. 
 \begin{figure}[h]
 	\centering
 	\begin{subfigure}[b]{0.49\linewidth}
 		\centering
 		\includegraphics[width=\linewidth]{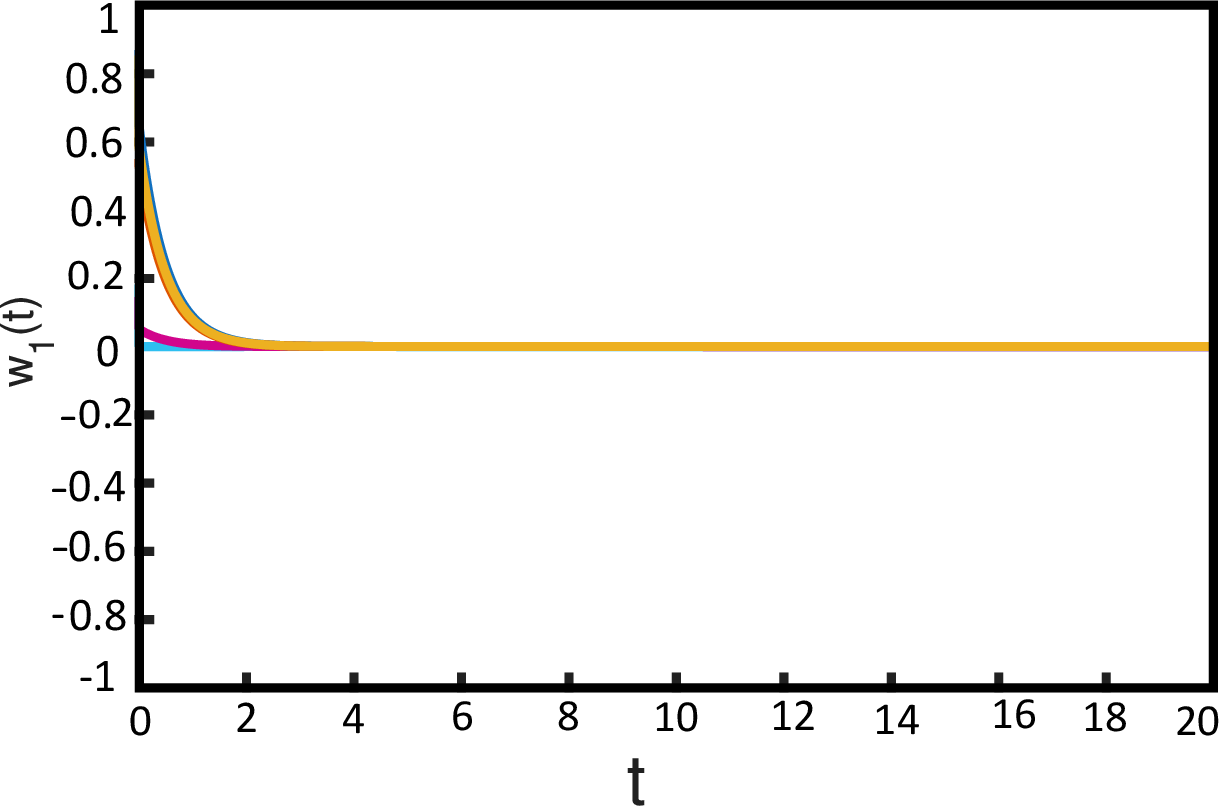}
 		\captionsetup{labelformat=empty}  
 		\caption{Figure (4a) : Transient behaviors of $w_1(t).$}
 		\label{fig:Fig4a}
 	\end{subfigure}
 	\hfill
 	\begin{subfigure}[b]{0.49\linewidth}
 		\centering
 		\includegraphics[width=\linewidth]{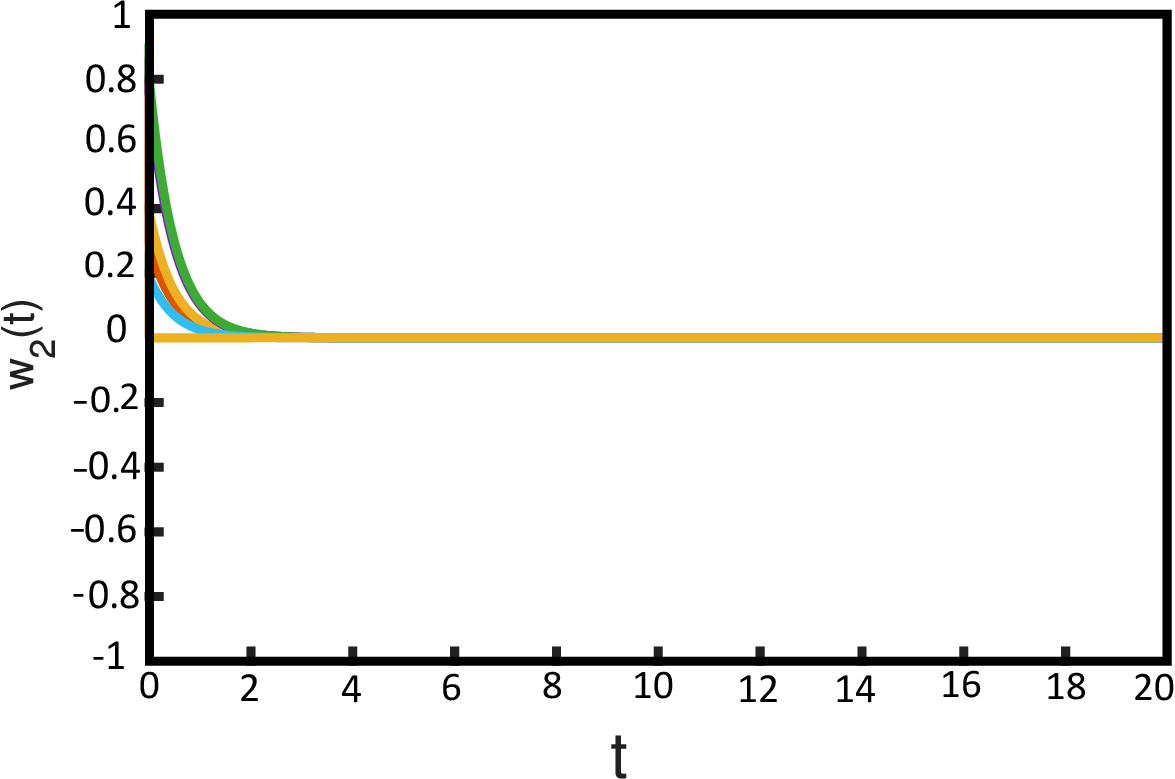}  
 		\captionsetup{labelformat=empty}
 		\caption{Figure (4b): Transient behaviors of $w_2(t).$}
 		\label{fig:Fig4b}
 	\end{subfigure}
 	
 \end{figure}

\par
Now, with fixed $\lambda = 10^{6},$ we examine the transient behavior of state variable $w(t)$ having initial point $w_{0} = (1.0, 1.0)$ of the neural network \eqref{eq:3} with decreasing value of the parameter  $\beta.$ It is observe that as $\beta$ approaches to zero,  the solution trajectories of the neural network get closer to the global minimum point (0,0) of the  problem  MPCC1 as exhibited in the Figures \eqref{fig:Fig5a} and \eqref{fig:Fig5b}. Further, the corresponding numerical results are summarized in Table \ref{tab:1}, which reports the approximate optimal solutions $(w_{1}, w_{2})$, their objective values $f(w_{1}^{}, w_{2}^{})$, and the associated absolute errors.
 \begin{figure}[h]
	\centering
	\begin{subfigure}[b]{0.49\linewidth}
		\centering
		\includegraphics[width=\linewidth]{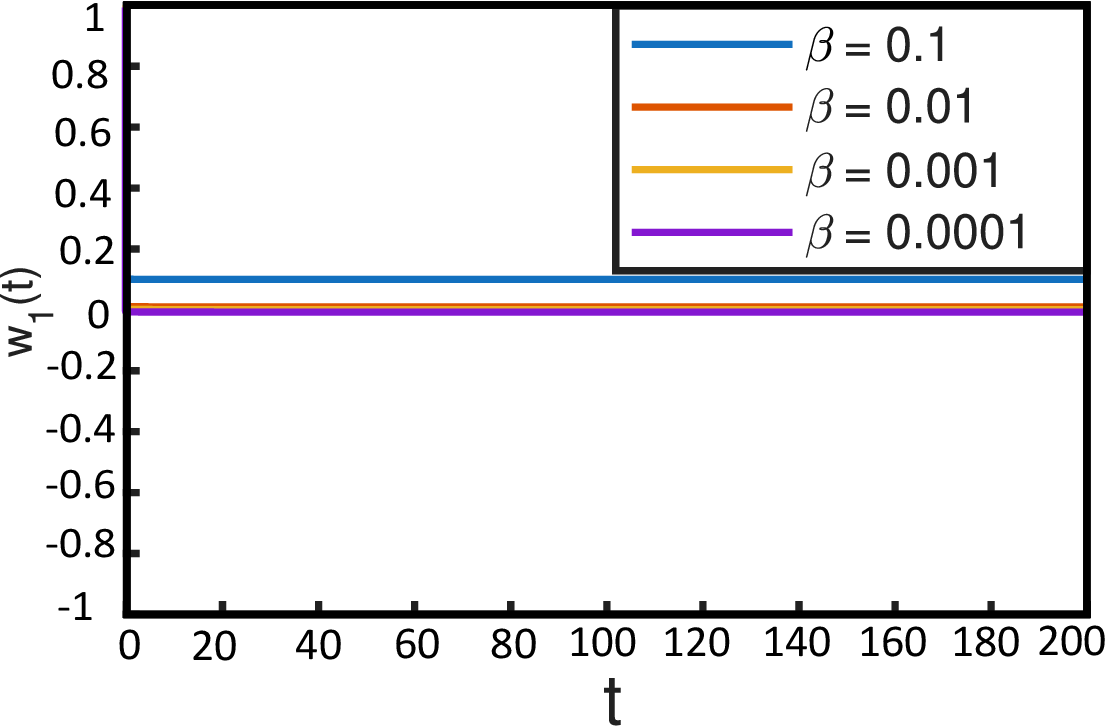}
		\captionsetup{labelformat=empty}  
		\caption{Figure (5a) : Transient behaviors of $w_1(t)$ with fixed $\lambda$.}
		\label{fig:Fig5a}
	\end{subfigure}
	\hfill
	\begin{subfigure}[b]{0.49\linewidth}
		\centering
		\includegraphics[width=\linewidth]{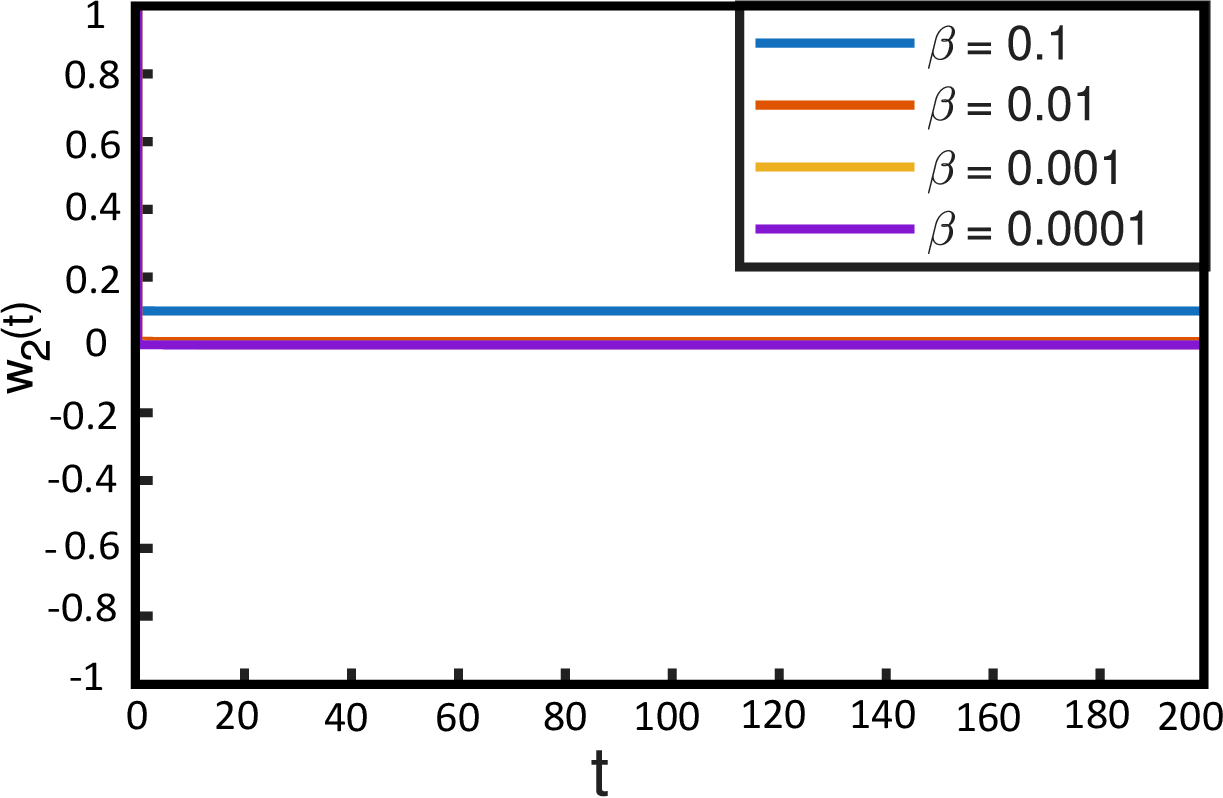}  
		\captionsetup{labelformat=empty}
		\caption{Figure (5b): Transient behaviors of $w_2(t)$ with fixed $\lambda$.}
		\label{fig:Fig5b}
	\end{subfigure}
\end{figure}
	\begin{table}[h!]
			\hspace*{-0.9cm}
	\centering
	\setlength{\tabcolsep}{1pt}
	\renewcommand{\arraystretch}{1.0}
	\begin{tabular}{|c|c|c|c|}
		\hline
		\textbf{$\beta$} & \makecell{Optimal solution\\[-0.7ex] $(w_1^*, w_2^*)$} & \makecell{Optimal value \\[-0.7ex] $f(w_1^*, w_2^*)$}
		& Absolute Error \\
		\hline
		0.1 & (0.100050000140542,   0.100050000140542) & 0 & 0  \\
		\hline
		0.01 & (0.0100500000194718,   0.0100500000194718 ) & 0 & 0 \\
		\hline
		0.001 & (0.00105000001965199,  0.00105000001965199) & 0 & 0 \\
		\hline
		0.0001 & (0.000150000019335945,   0.000150000019335945) & 0 & 0 \\
		\hline
	\end{tabular}
	\caption{}
	\label{tab:1}
\end{table}
\par
Next, we investigate the affects of  $\lambda$ on the convergence behavior of the state variable $w(t)$ having initial point $w_{0} = (1.5, 1.5)$ with fixed $\beta = 0.001.$ With increasing value of the parameter  $\lambda$, it is observe that the solution trajectories get closer to the global minimum point (0,0) as exhibited in the Figures \eqref{fig:Fig6a} and \eqref{fig:Fig6b}.  The corresponding numerical results are summarized in Table \ref{tab:2}.

\begin{figure}[h]
	\centering
	\begin{subfigure}[b]{0.46\linewidth}
		\centering
		\includegraphics[width=\linewidth]{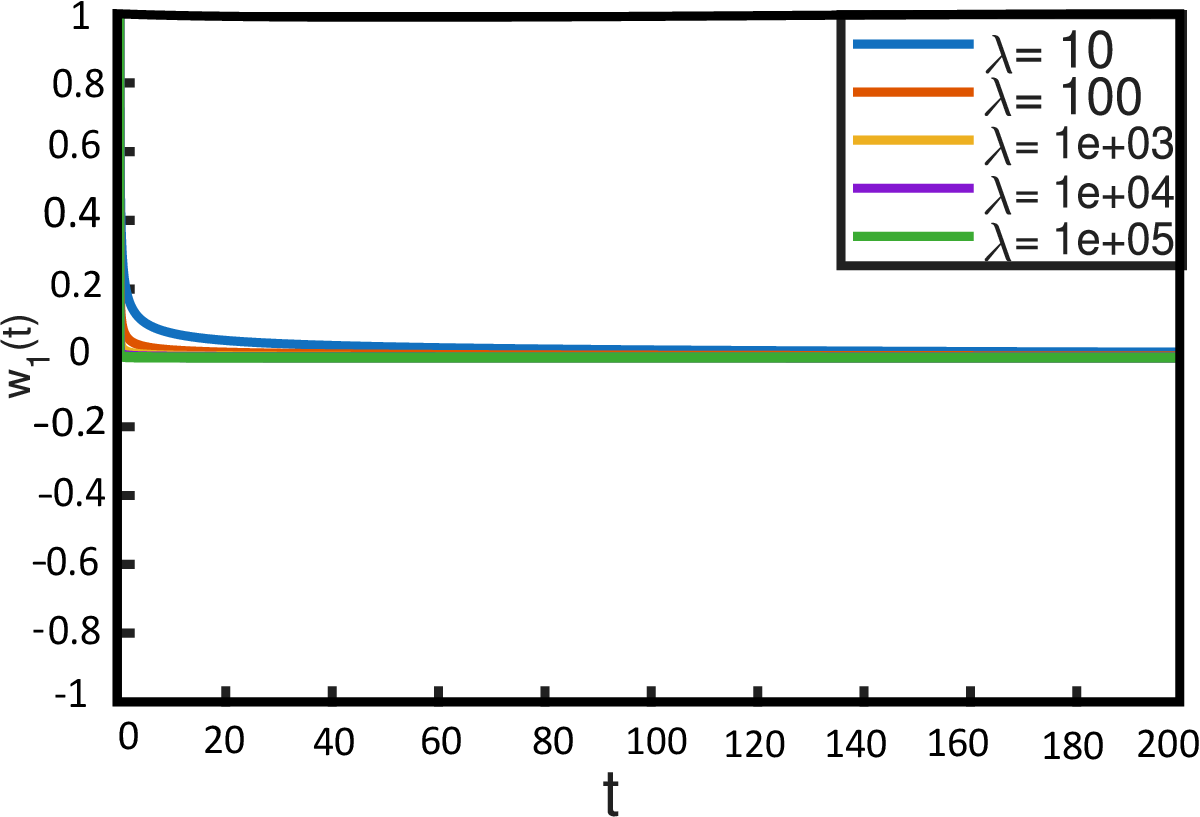}
		\captionsetup{labelformat=empty}  
		\caption{Figure (6a) :  Transient behaviors of $w_1(t)$ with fixed $\beta.$}
		\label{fig:Fig6a}
	\end{subfigure}
	\hfill
	\begin{subfigure}[b]{0.49\linewidth}
		\centering
		\includegraphics[width=\linewidth]{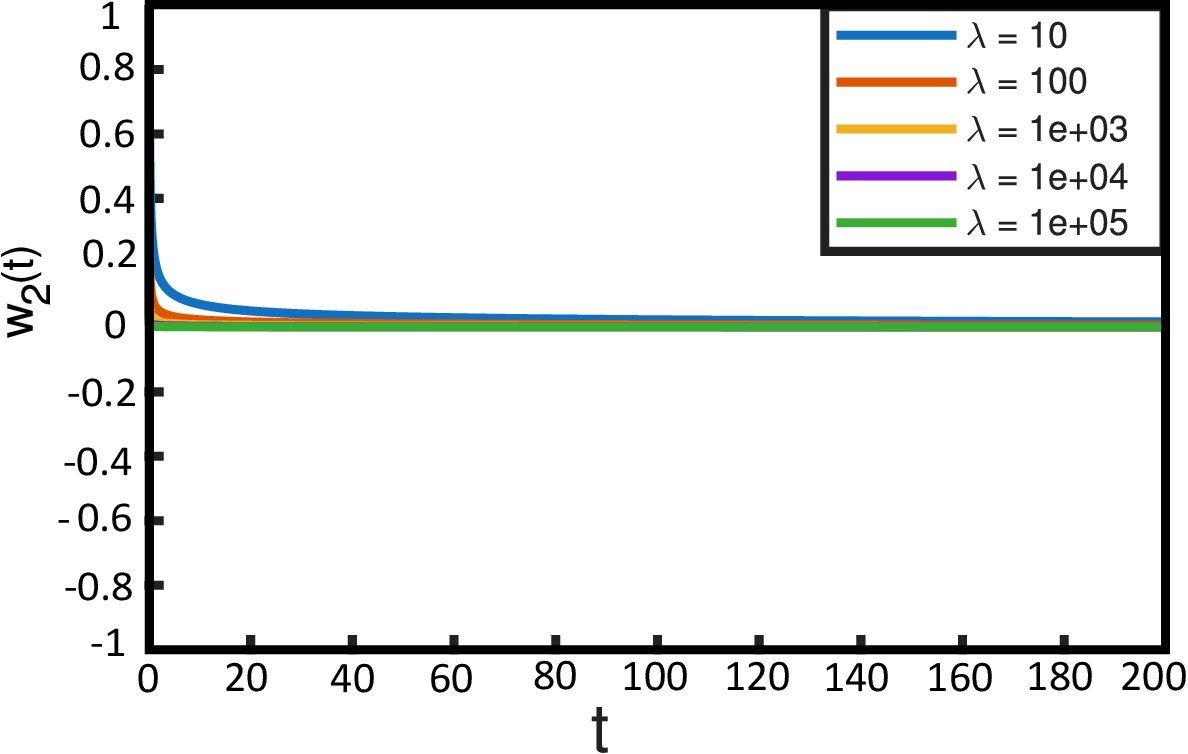}  
		\captionsetup{labelformat=empty}
		\caption{Figure (6b): Transient behaviors of $w_2(t)$ with fixed $\beta$.}
		\label{fig:Fig6b}
	\end{subfigure}
\end{figure}
\begin{table}[h!] 
		\hspace*{-0.9cm}
	\centering
	\setlength{\tabcolsep}{1pt}
	\renewcommand{\arraystretch}{1.0}
	\begin{tabular}{|c|c|c|c|}
		\hline
		\textbf{$\lambda$} & \makecell{Optimal solution\\[-0.7ex] $(w_1^*, w_2^*)$} & \makecell{Optimal value \\[-0.7ex] $f(w_1^*, w_2^*)$}
		& Absolute Error \\
		\hline
		10 & (0.0168105088432284,   0.0168105088432284) & 0 & 0 \\
		\hline
	100 & (0.00599997220798751   0.00599997220798751) & 0 & 0 \\
		\hline
		1000 & (0.00258113797276159   0.00258113797276159) & 0 & 0 \\
		\hline
		10000 & (0.00149999999698513   0.00149999999698513) & 0 & 0 \\
		\hline
		100000 & (0.00115811390460986,   0.00115811390460986) & 0 & 0\\
		\hline
	\end{tabular}
	\caption{}
	\label{tab:2}
\end{table}
\end{Example} 
\begin{Example}\cite{39} \label{ex:4}	Consider the  problem:
	$$
	\begin{aligned}
	\text{(MPCC3)} \quad	\min\quad & (w_1 - 1)^2 + \left(w_2 - \tfrac{1}{2}\right)^2 \\
		\text{subject to}\quad 
		& w_1 \le 1,~w_2 \ge 0,\\
		& 0 \le 2w_1 + w_2  \perp 2 - (w_1 - 1)^2 - (w_2 - 1)^2 \ge 0 .
	\end{aligned}
	$$
	\par In order to  solve the problem MPCC3, we employ the proposed neural network \eqref{eq:3}. Using the above mentioned mathematical tools,  the numerical experiments with 10 random initial points,  shows that the solution trajectories $w_1(t)$ and $w_2(t)$ of the  neural network \eqref{eq:3}   converge  to the global minimum point  $(0,0)$ with the parameters $\lambda=10^{6}$ and $\beta=0.00001$ as depicted  in the figures \eqref{fig:Fig13a} and \eqref{fig:Fig13b}. 
	\begin{figure}[h]
		\centering
		\begin{subfigure}[b]{0.49\linewidth}
			\centering
			\includegraphics[width=\linewidth]{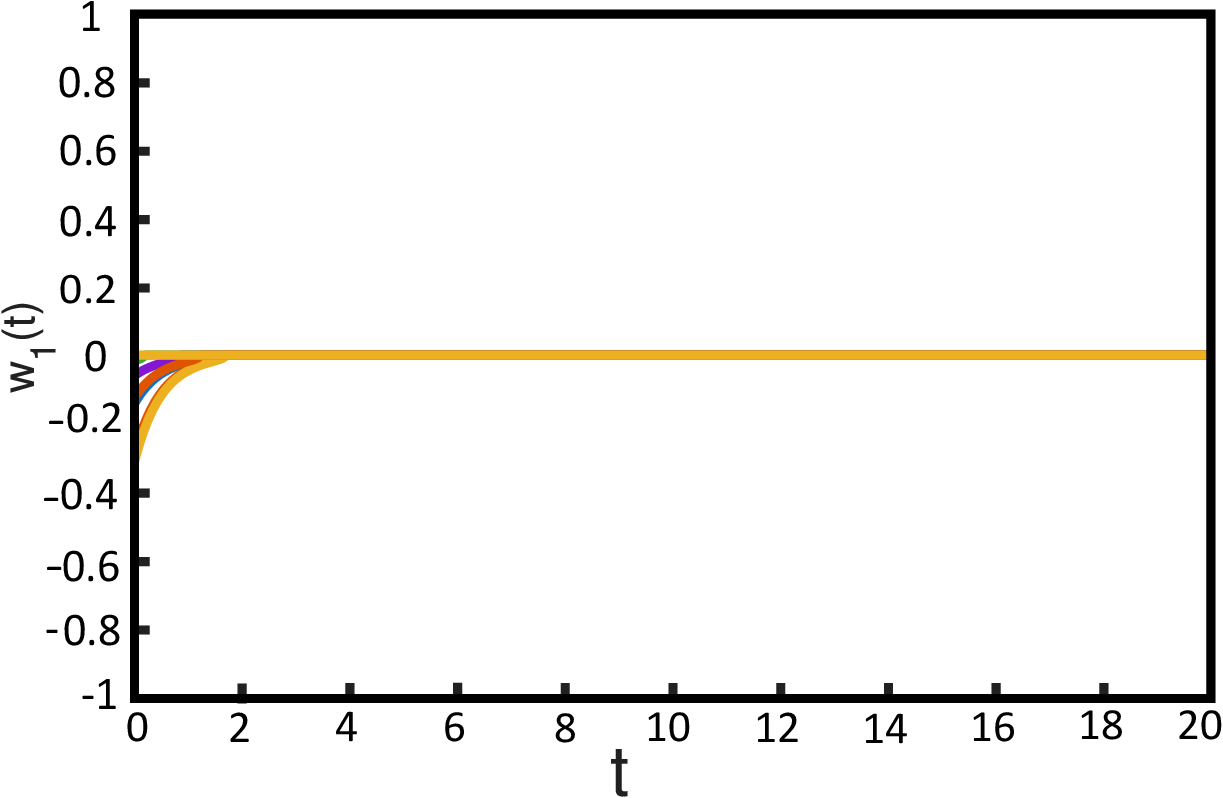}
			\captionsetup{labelformat=empty}  
			\caption{Figure (7a) : Transient behaviors of $w_1(t).$}
			\label{fig:Fig13a}
		\end{subfigure}
		\hfill
		\begin{subfigure}[b]{0.49\linewidth}
			\centering
			\includegraphics[width=\linewidth]{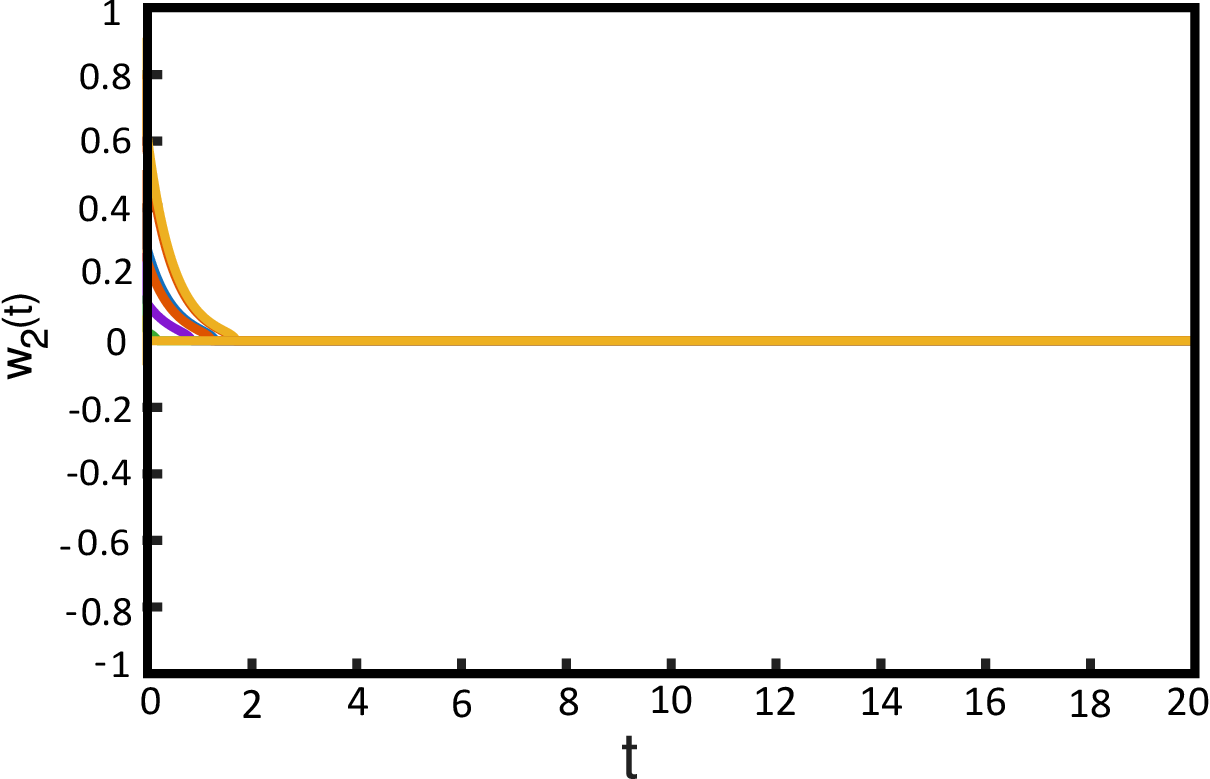}  
			\captionsetup{labelformat=empty}
			\caption{Figure (7b): Transient behaviors of $w_2(t).$}
			\label{fig:Fig13b}
		\end{subfigure}
	\end{figure} 
	\par Now, with fixed $\lambda = 10^{6},$ we examine the transient behavior of state variable $w(t)$ having initial point $w_{0} = (1.0, 1.0)$ of the neural network \eqref{eq:3} with decreasing value of the parameter  $\beta.$ It is observe that as $\beta$ approaches to zero,  the solution trajectories of the neural network get closer to the global minimum point (0,0) of the  problem  MPCC3 as exhibited in the Figures  \eqref{fig:Fig14a} and \eqref{fig:Fig14b}. Further, the corresponding numerical results are summarized in Table \ref{tab:7}, which reports the approximate optimal solutions $(w_{1}, w_{2})$, their objective values $f(w_{1}^{}, w_{2}^{})$, and the associated absolute errors. 
	\par
	\begin{figure}[h]
		\centering
		\begin{subfigure}[b]{0.49\linewidth}
			\centering
			\includegraphics[width=\linewidth]{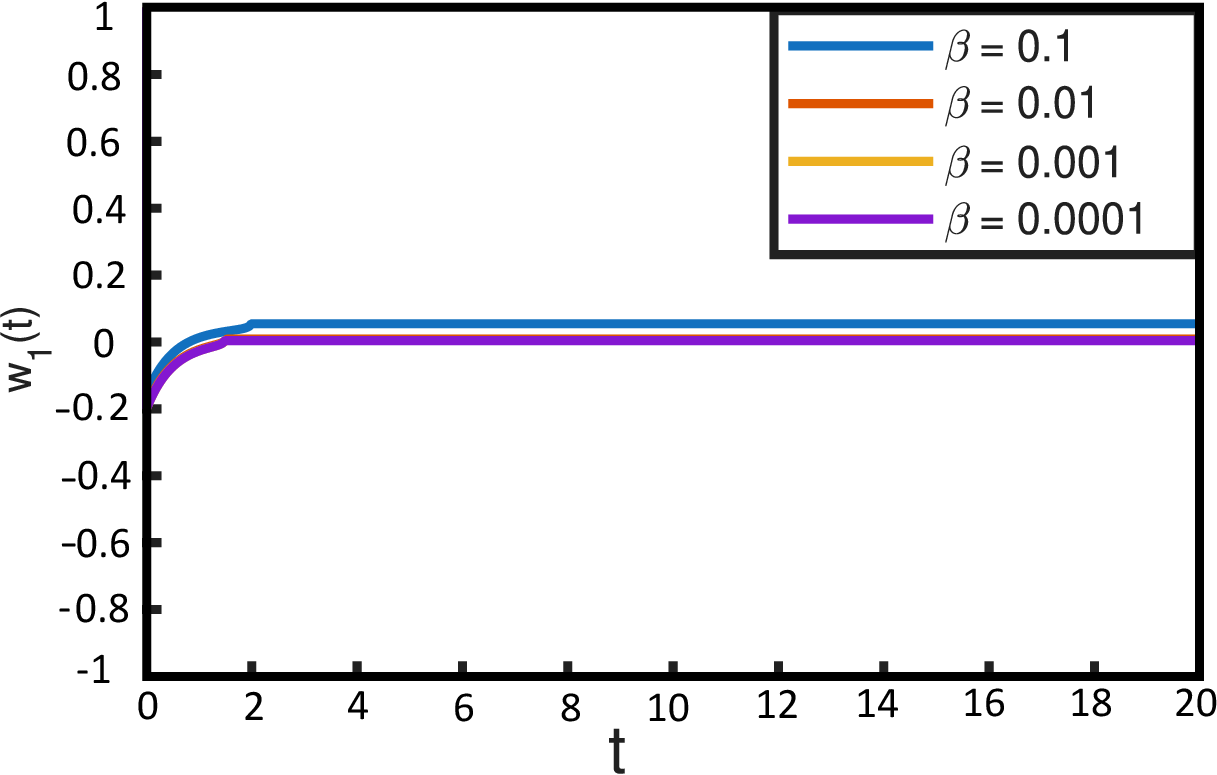}
			\captionsetup{labelformat=empty}  
			\caption{Figure (8a) : Transient behaviors of $w_1(t)$ with fixed $\lambda.$}
			\label{fig:Fig14a}
		\end{subfigure}
		\hfill
		\begin{subfigure}[b]{0.49\linewidth}
			\centering
			\includegraphics[width=\linewidth]{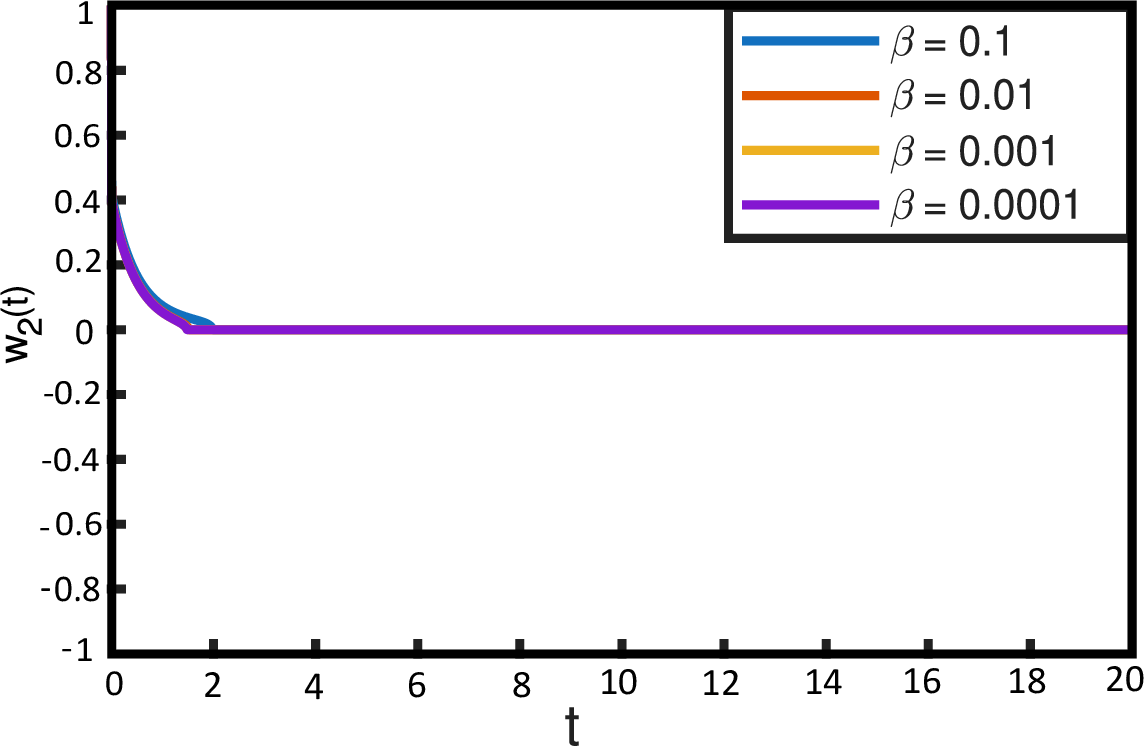}  
			\captionsetup{labelformat=empty}
			\caption{Figure (8b): Transient behaviors of $w_2(t)$ with fixed $\lambda.$}
			\label{fig:Fig14b}
		\end{subfigure}
	\end{figure}
	\begin{table}[h!]
			\hspace*{-1cm}
		\centering
		\setlength{\tabcolsep}{1pt}
		\renewcommand{\arraystretch}{1.0}
		\begin{tabular}{|c|c|c|c|}
			\hline
			\textbf{$\beta$} & \makecell{Optimal solution\\[-0.7ex] $(w_1^*, w_2^*)$} & \makecell{Optimal value \\[-0.7ex] $f(w_1^*, w_2^*)$}
			& Absolute Error \\
			\hline
			0.1 & (0.0547347728730422,   -5.57362581295502e-07) & 1.14352690697827 & 0.106473093021729 \\
			\hline
			0.01 & 0.00898313645704711,  -4.95784188958737e-07) & 1.23211491961095 & 0.0178850803890533 \\
			\hline
			0.001 & (0.00447091291113772,   -4.97282434599503e-07) & 1.24107866052267 & 0.00892133947733464 \\
			\hline
			0.0001 & (0.00402025441315761,   -4.97464725953594e-07) & 1.2419761510842  & 0.00802384891579511 \\
			\hline
		\end{tabular}
		\caption{}
		\label{tab:7}
	\end{table}
	\par Next, we investigate the affects of  $\lambda$ on the convergence behavior of the state variable $w(t)$ having initial point $w_{0} = (1.0, 1.0)$ with fixed  $\beta = 0.000001.$ With increasing value of the parameter  $\lambda$, it is observe that the solution trajectories get closer to the global minimum point (0,0) as exhibited in the Figures \eqref{fig:Fig15a} and \eqref{fig:Fig15b}.  The corresponding numerical results are summarized in Table \ref{tab:8}.
	\begin{figure}[h]
		\centering
		\begin{subfigure}[b]{0.49\linewidth}
			\centering
			\includegraphics[width=\linewidth]{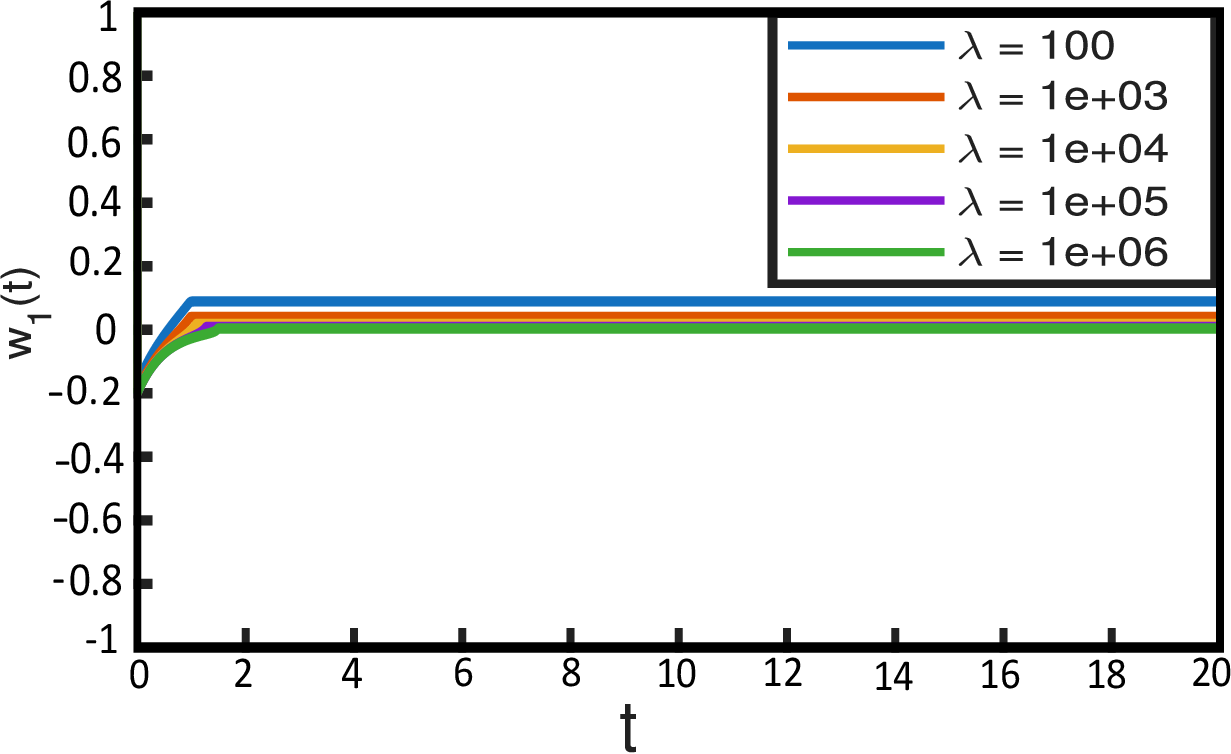}
			\captionsetup{labelformat=empty}  
			\caption{Figure (9a) : Transient behaviors of $w_1(t)$ with fixed $\beta.$}
			\label{fig:Fig15a}
		\end{subfigure}
		\hfill
		\begin{subfigure}[b]{0.46\linewidth}
			\centering
			\includegraphics[width=\linewidth]{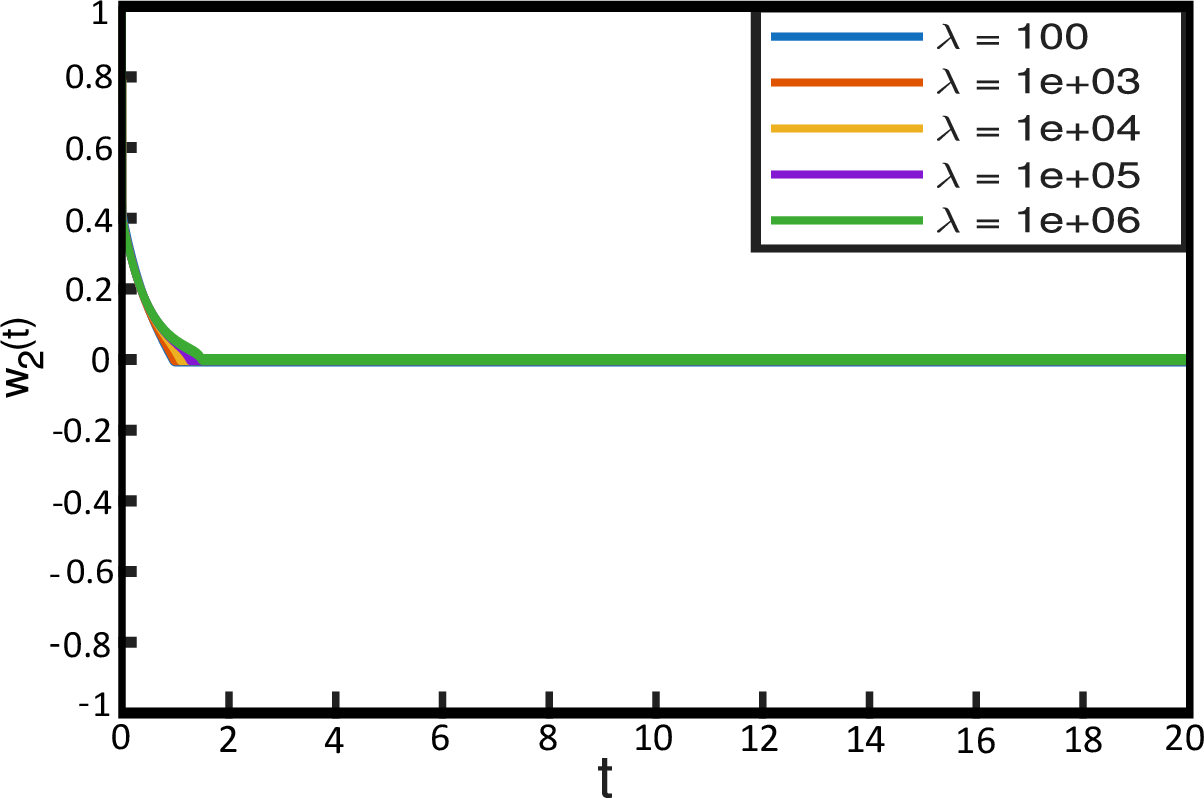}  
			\captionsetup{labelformat=empty}
			\caption{Figure (9b): Transient behaviors of $w_2(t)$ with fixed $\beta.$}
			\label{fig:Fig15b}
		\end{subfigure}
	\end{figure}
	\begin{table}[h!] 
			\hspace*{-1cm}
		\centering
		\setlength{\tabcolsep}{1pt}
		\renewcommand{\arraystretch}{1.0}
		\begin{tabular}{|c|c|c|c|}
			\hline
			\textbf{$\lambda$} & \makecell{Optimal solution\\[-0.7ex] $(w_1^*, w_2^*)$} & \makecell{Optimal value \\[-0.7ex] $f(w_1^*, w_2^*)$}
			 & Absolute Error \\
			\hline
			100 & (0.0895849855530089,   -0.00445364859562869) & 1.08332898211196 & 0.166671017888043 \\
			\hline
			1000 & (0.0401832807892819,   -0.000475660458775592) & 1.17172402118807 &  0.0782759788119258 \\
			\hline
			10000 & (0.0184902830395593,   -4.88662867458115e-05) & 1.21341019316242 & 0.0365898068375758 \\
			\hline
			100000 & (0.00856469689824011,   -4.94705982203277e-06) & 1.23294890732077 & 0.017051092679226\\
			\hline
			1000000 & ( 0.00397519417035708,   -4.97510345691846e-07) & 1.24206591133857 & 0.00793408866142897 \\
			\hline
		\end{tabular}
		\caption{}
		\label{tab:8}
	\end{table}

	\newpage Now, we   present  comparative study between the proposed neural network \eqref{eq:3} and the  penalty algorithms ($l_1$, lower-order, and quadratic) for MPCCs reported in 	Yang and Huang \cite{27}.  The comparison is conducted to evaluate the effectiveness of the proposed method in terms of solution numerical accuracy and solution reliability.
	\begin{Example}\cite{27}\label{ex:5} 	Consider the  problem:
		$$
		\begin{aligned}
				\text{(MPCC4)} \quad	\min\quad & (w_1 + 1)^2 + (w_2 - 2.5)^2 + (w_3 + 1)^2 \\
				\text{subject to}~
			&w_1 \ge 0,~w_3 \ge 0, ~ -e^{w_1} + w_2 - e^{w_3} \ge 0,\\
			& w_1\bigl(-e^{w_1} + w_2 - e^{w_3}\bigr) = 0.
		\end{aligned}
		$$
		\par The transient  behaviors  of the state variables $w_1(t), w_2(t)$ and $w_3(t)$ at  $ \beta = 0.01$ and $\lambda=10^6,$  with 10 random  initial points are exhibited in the Figures~\eqref{fig:Fig16a}, \eqref{fig:Fig16b} and \eqref{fig:Fig16c}, respectively. Using the above mentioned  mathematical tools, the computer simulation shows that  the trajectories of the  neural network \eqref{eq:3}  converges to  approximate optimal solution $w^*=(-0.000002, 2.500000,-0.000002)$ of the problem MPCC4. Further, the quantitative comparison between the proposed neural network \eqref{eq:3} and penalty algorithms ($l_1$, lower-order, and quadratic) given  by Yang and Huang  \cite{27} are presented in Table \ref{tab:9}. It is observed that both the methods demonstrate convergence toward near-optimal solutions. However, the proposed neural  network \eqref{eq:3} exhibits the numerical accuracy and solution reliability.
		\begin{figure}[h]
			\centering
			\begin{subfigure}[b]{0.49\linewidth}
				\centering
				\includegraphics[width=\linewidth]{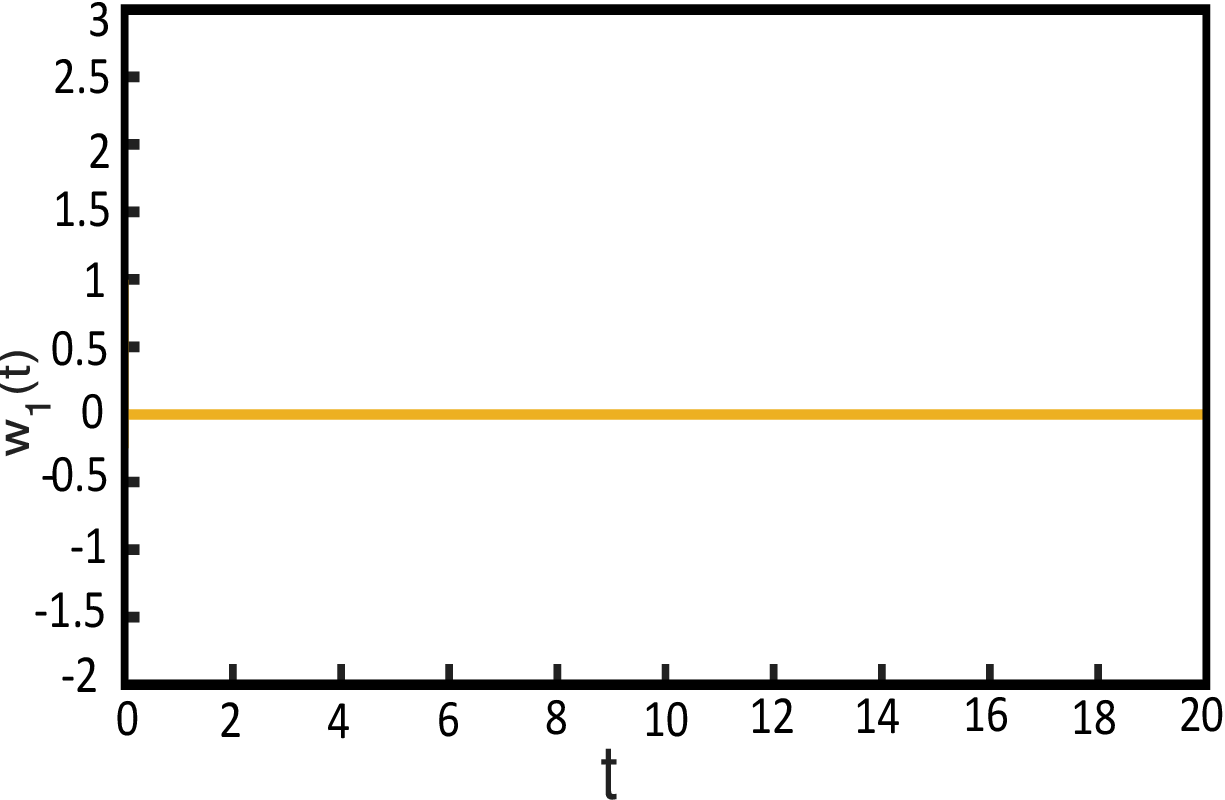}
				\captionsetup{labelformat=empty}  
				\caption{Figure (10a) : Transient behaviors of $w_1(t).$}
				\label{fig:Fig16a}
			\end{subfigure}
			\hfill
			\begin{subfigure}[b]{0.49\linewidth}
				\centering
				\includegraphics[width=\linewidth]{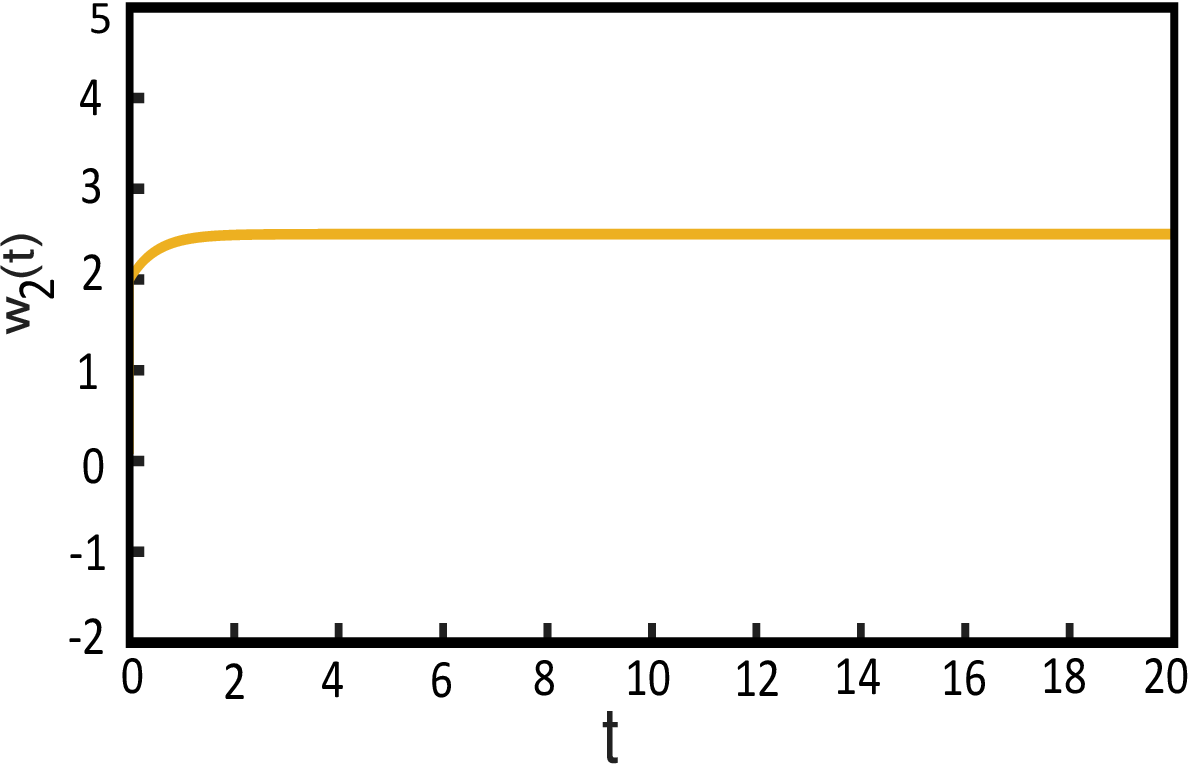}  
				\captionsetup{labelformat=empty}
				\caption{Figure (10b): Transient behaviors of $w_2(t).$}
				\label{fig:Fig16b}
			\end{subfigure}
			\\
			\vspace{1.1em}
			\begin{subfigure}[b]{0.49\linewidth}
				\centering
				\includegraphics[width=\linewidth]{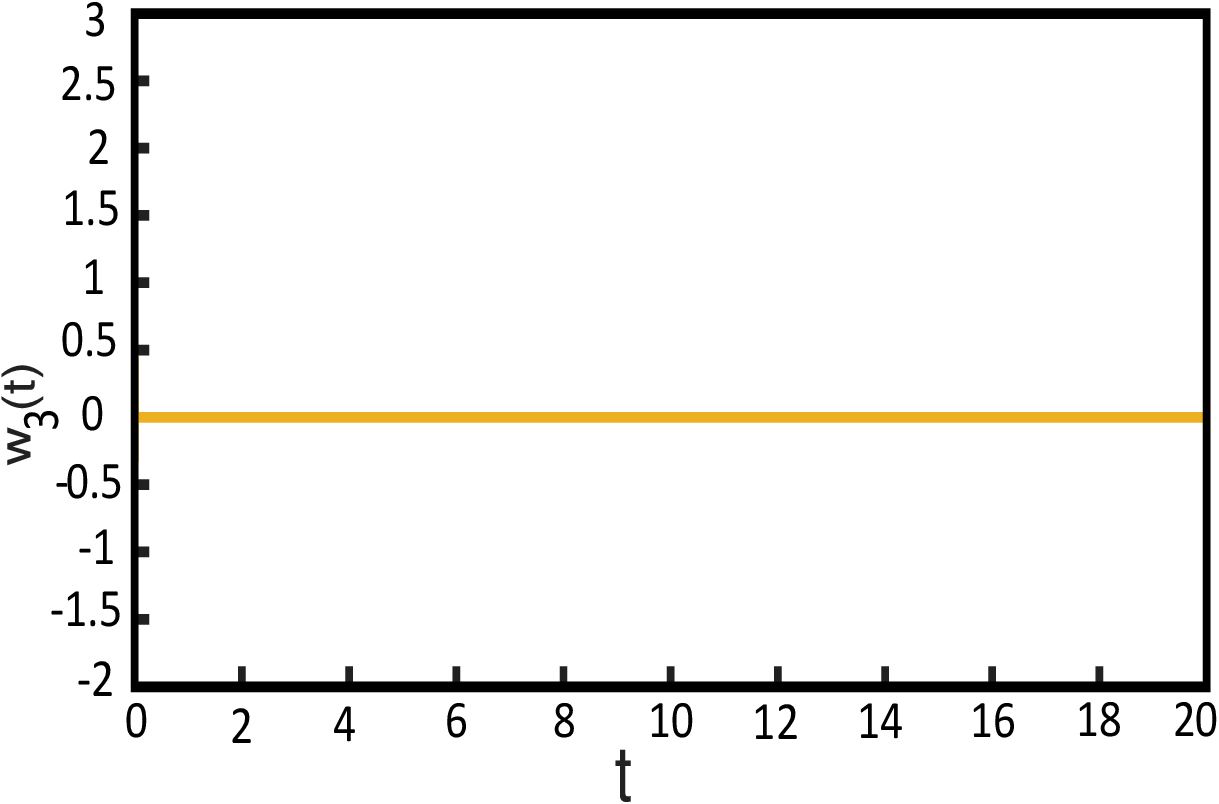}  
				\captionsetup{labelformat=empty}
				\caption{Figure (10c): Transient behaviors of $w_3(t).$}
				\label{fig:Fig16c}
			\end{subfigure}
			
		\end{figure}
	\end{Example}

\begin{table}[h!]
	\hspace*{-1cm}
	\centering
	\setlength{\tabcolsep}{4pt}
	\renewcommand{\arraystretch}{1.0}
	\begin{tabular}{|c|c|c|c|}
		\hline
	Method &	Intial points  & \makecell{Approximate optimal solution\\[-0.7ex] $(w_1^*, w_2^*,w_3^*)$} & \makecell{Approximate optimal value \\[-0.7ex] $f(w_1^*, w_2^*,w_3^*)~$}
	 \\ 
		\hline
		\makecell{$l_1$\\[-0.7ex] penalty algorithm} 	& (1, 1, 1) &
		$(0,2.5,0)$ &
		2 
	 \\
		\hline
\makecell{Lower-order\\[-0.7ex] penalty algorithm} 	&	(1, 1, 1) &
		$(0, 2.5, 0)$ &
		2\\
		\hline
	\makecell{Quadratic \\[-0.7ex] penalty algorithm} 	& (1, 1, 1) & $(-0.0001, 2.5, -0.0001)$ &
		1.9996 \\
		\hline
		\makecell{Proposed\\[-0.7ex] neural network} 
	&	10 random &
		$(-0.000002, 2.500000, -0.000002)$ &
		1.99999  \\
		\hline
	\end{tabular}
	\caption{}
	\label{tab:9}
\end{table}
\end{Example}
	\begin{Example}\cite{27}\label{ex:6} 	Consider the  problem:
	$$
	\begin{aligned}
		\text{(MPCC5)} \quad	\min\quad & (w_1 + 1)^2 + w_2^2 + 10 (w_3 - 1)^2 \\
		\text{subject to}~
		&w_1 \ge 0,~w_3 \ge 0, ~ -e^{w_1} + w_2 - e^{w_3} \ge 0, \\
		& w_1\bigl(-e^{w_1} + w_2 - e^{w_3}\bigr) = 0.
	\end{aligned}
	$$
		\par The transient  behaviors  of the state variables $w_1(t), w_2(t)$ and $w_3(t)$ at  $ \beta = 0.01$ and $\lambda=10^5,$  with 10 random  initial points are exhibited in the Figures~\eqref{fig:Fig17a}, \eqref{fig:Fig17b} and \eqref{fig:Fig17c}, respectively. Using the above mentioned  mathematical tools, the computer simulation shows that  the trajectories of the  neural network \eqref{eq:3}  converges to  approximate optimal solution $w^* = (-0.000074, 2.709987, 0.536561)$ of the problem MPCC4. Further, the quantitative comparison between the proposed neural network \eqref{eq:3} and penalty algorithms ($l_1$, lower-order, and quadratic) given  by Yang and Huang  \cite{27} are presented in Table  \ref{tab:10}. It is observed that both the methods demonstrate convergence toward near-optimal solutions. However, the proposed neural  network \eqref{eq:3} exhibits the numerical accuracy and solution reliability.
	\begin{figure}[h]
		\centering
		\begin{subfigure}[b]{0.49\linewidth}
			\centering
			\includegraphics[width=\linewidth]{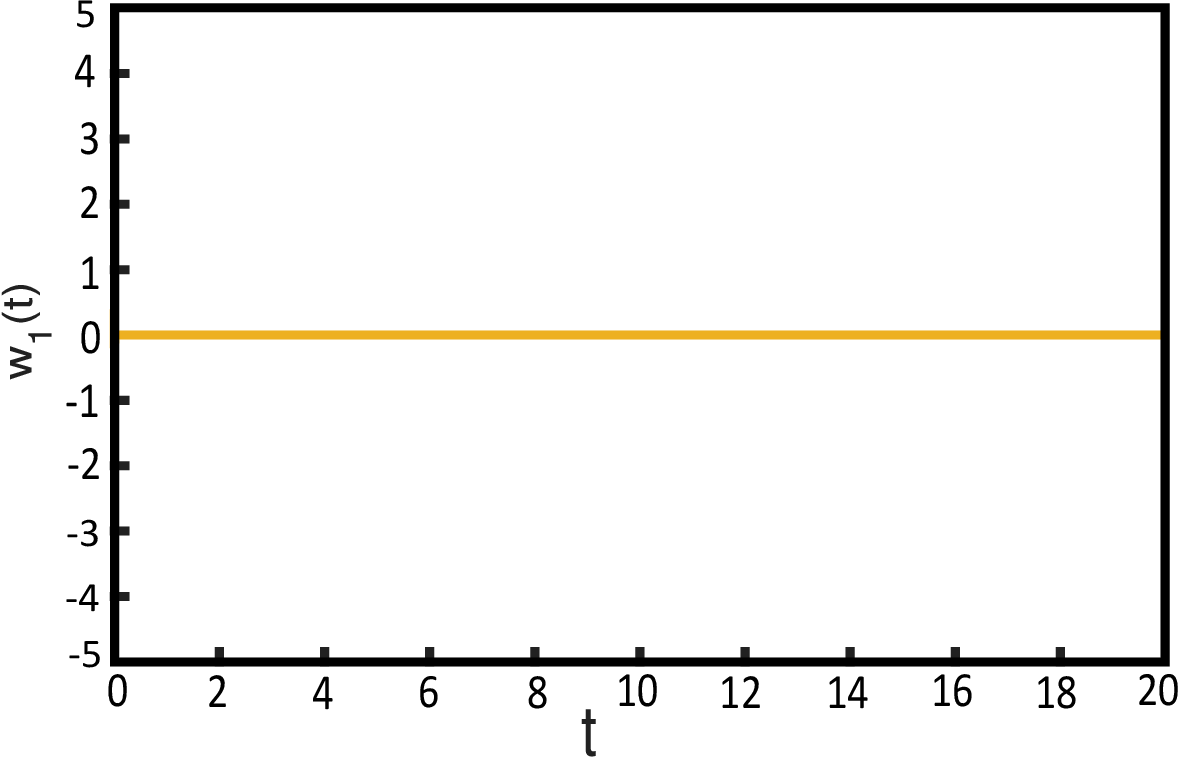}
			\captionsetup{labelformat=empty}  
			\caption{Figure (11a) : Transient behaviors of $w_1(t).$}
			\label{fig:Fig17a}
		\end{subfigure}
		\hfill
		\begin{subfigure}[b]{0.49\linewidth}
			\centering
			\includegraphics[width=\linewidth]{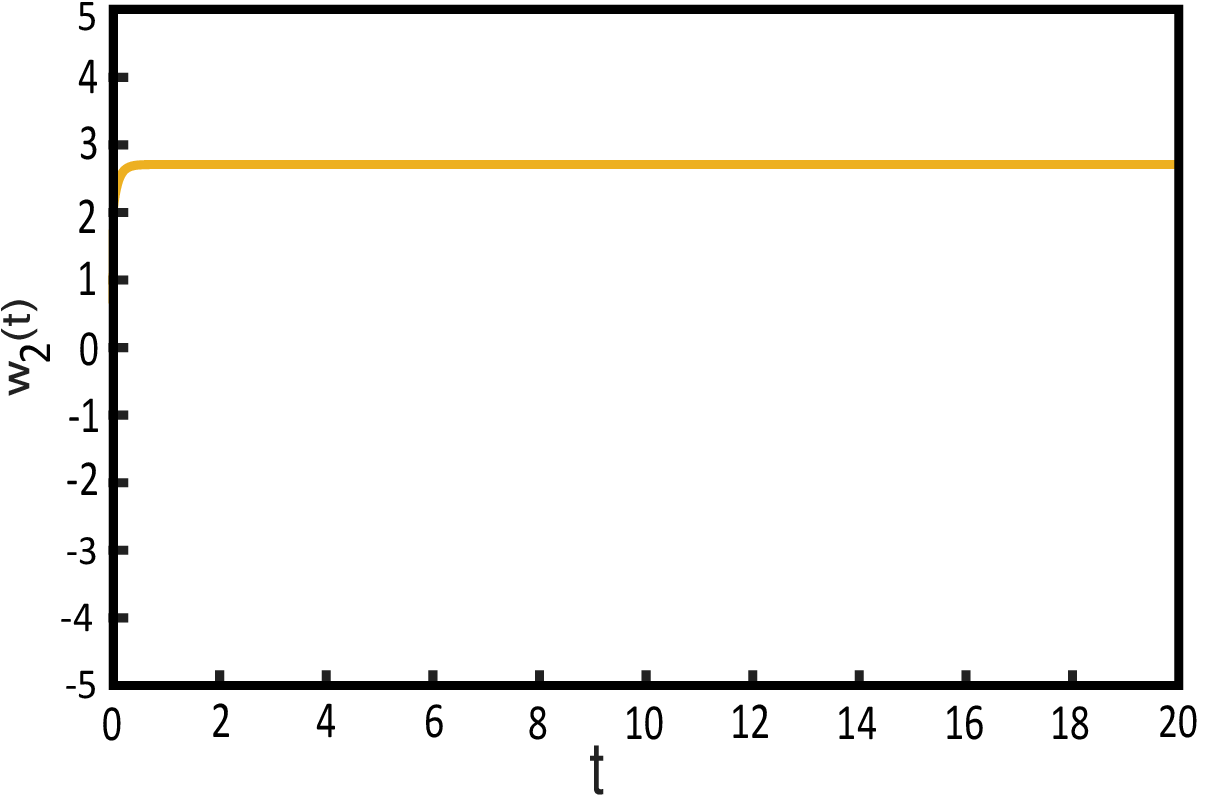}  
			\captionsetup{labelformat=empty}
			\caption{Figure (11b): Transient behaviors of $w_2(t).$}
			\label{fig:Fig17b}
		\end{subfigure}
		\\
		\vspace{1.1em}
		\begin{subfigure}[b]{0.49\linewidth}
			\centering
			\includegraphics[width=\linewidth]{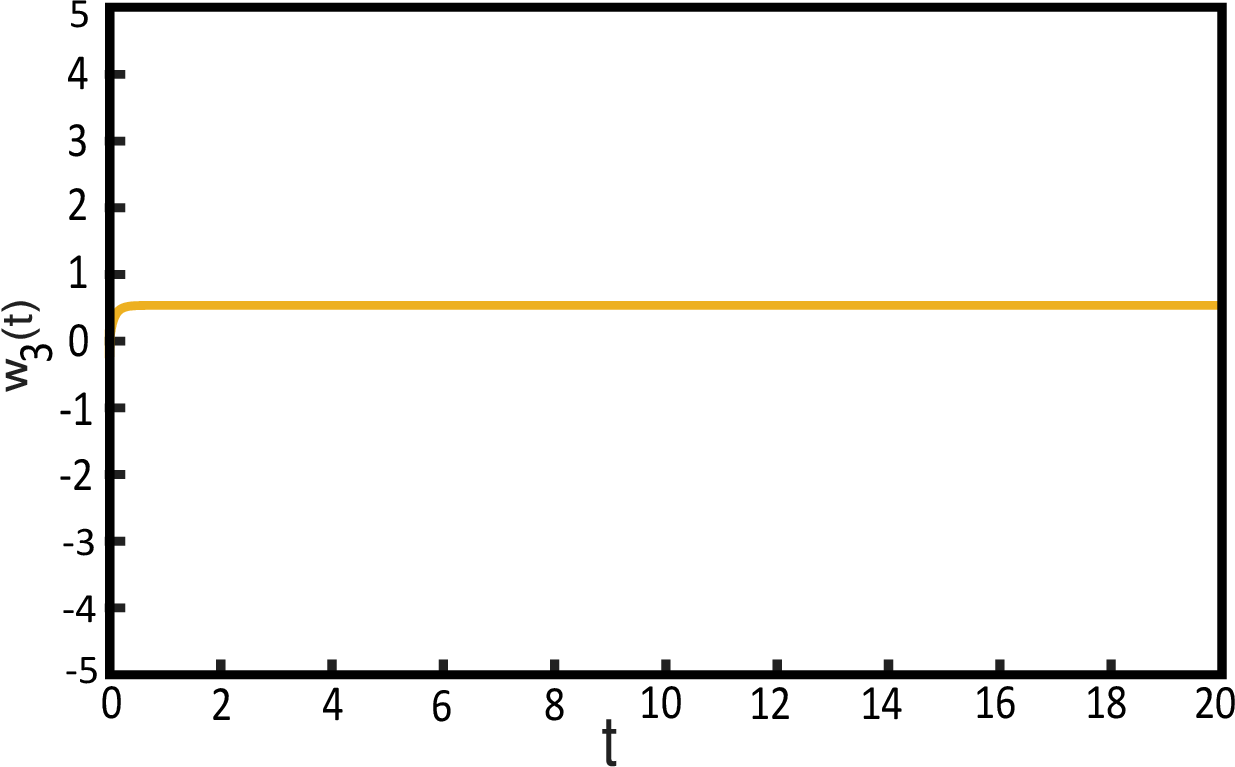}  
			\captionsetup{labelformat=empty}
			\caption{Figure (11c): Transient behaviors of $w_3(t).$}
			\label{fig:Fig17c}
		\end{subfigure}
		
	\end{figure}
\begin{table}[h!]
	\hspace*{-1cm}
	\centering
	\setlength{\tabcolsep}{4pt}
	\renewcommand{\arraystretch}{1.0}
	\begin{tabular}{|c|c|c|c|}
		\hline
		Method &	Intial points  & \makecell{Approximate optimal solution\\[-0.7ex] $(w_1^*, w_2^*,w_3^*)$} & \makecell{Approximate optimal value \\[-0.7ex] $f(w_1^*, w_2^*,w_3^*)~$} \\ 
		\hline
		 \makecell{$l_1$\\[-0.7ex] penalty algorithm} 
		 & (1, 1, 1) &
		$ (0,2.7102,0.5366)$ &
		 10.4923
		\\
		\hline
		\makecell{Lower-order\\[-0.7ex] penalty algorithm} 
		 &	(1, 1, 1) &
		$ (0, 2.7091, 0.536)$ &
		 10.4925\\
		\hline \makecell{Quadratic \\[-0.7ex] penalty algorithm} 
		 	& (1, 1, 1) & $ (0, 2.7111, 0.5372)$ &
		 10.4924\\
		\hline
		\makecell{Proposed\\[-0.7ex] neural network}	&	10 random &
		$(-0.000074, 2.709987, 0.536561)$ &
		10.491639  \\
		\hline
	\end{tabular}
	\caption{}
	\label{tab:10}
\end{table}
\end{Example}
	\begin{Example}\cite{27}\label{ex:7} 	Consider the  problem:
	$$
	\begin{aligned}
		\text{(MPCC6)}\quad	\min \quad & \frac12\Big[(w_1-3)^2+(w_2-4)^2\Big] \\
		\text{Subject to}\quad 
		& w_1,w_2,w_3,w_4 \ge 0,\quad 0\le w_5 \le 10,\\
		& (1+0.2w_5)w_1-(1+1.333w_5)-0.333w_3+2w_1w_4 \ge 0,\\
		& w_1\Big((1+0.2w_5)w_1-(1+1.333w_5)-0.333w_3+2w_1w_4\Big)=0,\\
		& (1+0.1w_5)w_2 - w_5 + w_3 + 2w_2w_4 \ge 0,\\
		& w_2\Big((1+0.1w_5)w_2 - w_5 + w_3 + 2w_2w_4\Big)=0,\\
		& 0.333w_1 - w_2 + 1 - 0.1w_5 \ge 0,\\
		& w_3\big(0.333w_1 - w_2 + 1 - 0.1w_5\big)=0,\\
		& 9+0.1w_5 - w_1^{2}-w_2^{2} \ge 0,\\
		& w_4\big(9+0.1w_5 - w_1^{2}-w_2^{2}\big)=0.
	\end{aligned}
	$$
		\par The transient  behaviors  of the state variables  $w_1(t), w_2(t), w_3(t), w_4(t)$ and $w_5(t)$ at  $ \beta = 0.00001$ and $\lambda=10^4,$  with 10 random  initial points are exhibited in the Figures~\eqref{fig:Fig18a}, \eqref{fig:Fig18b}, \eqref{fig:Fig18c}, \eqref{fig:Fig18d} and \eqref{fig:Fig18e}, respectively. Using the above mentioned  mathematical tools, the computer simulation shows that  the trajectories of the  neural network \eqref{eq:3}  converges to  approximate optimal solution $w^* = (2.552739, 1.640964, -0.000039, 0.032896,$$\\ 2.092187)$ of the problem MPCC6. Further, the quantitative comparison between the proposed neural network \eqref{eq:3} and penalty algorithms ($l_1$, lower-order, and quadratic) given  by Yang and Huang  \cite{27} are presented in Table \ref{tab:10}. It is observed that both the methods demonstrate convergence toward near-optimal solutions. However, the proposed neural  network \eqref{eq:3} exhibits the numerical accuracy and solution reliability. Also, for the problem MPCC6,  Outrata \cite{9} obtained the optimal value  as 
		3.2077, while the proposed neural network  \eqref{eq:3} gives the optimal value   2.88254.
	\begin{figure}[!htbp]
		\centering
		\begin{subfigure}[b]{0.47\linewidth}
			\centering
			\includegraphics[width=\linewidth]{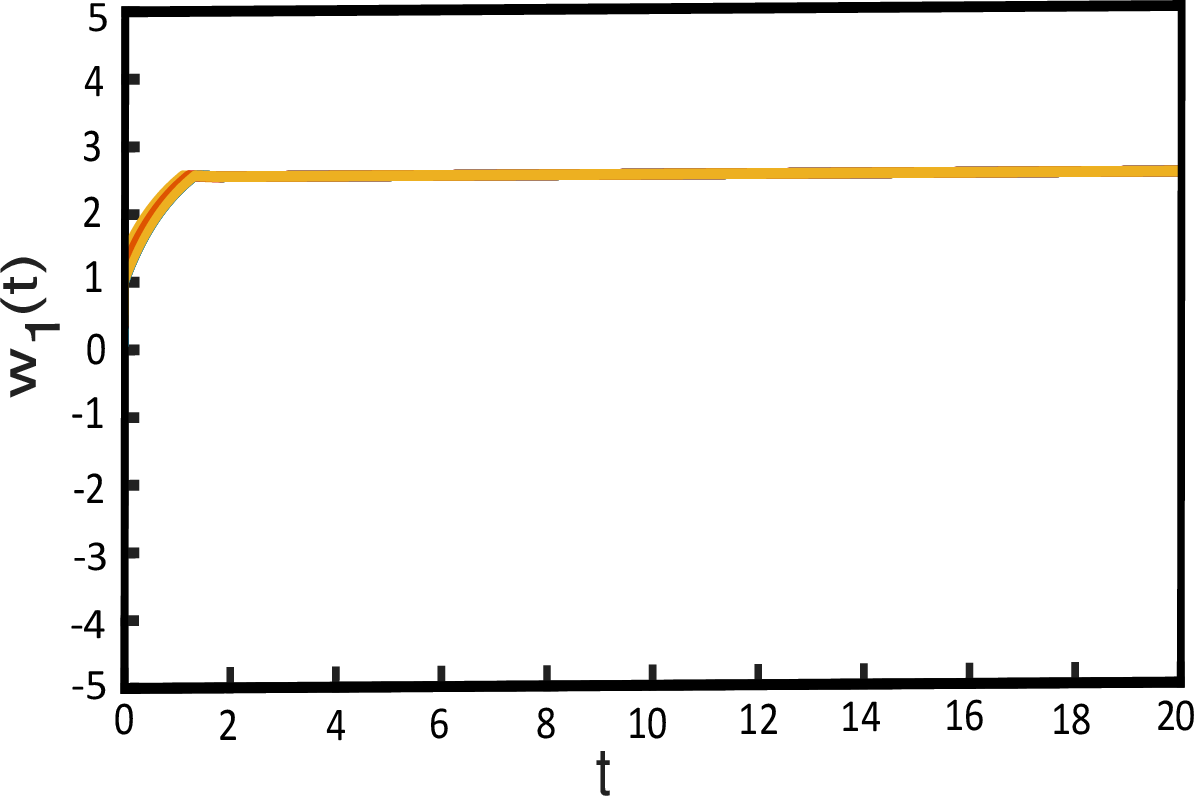}
			\captionsetup{labelformat=empty}
			\caption{Figure (12a): Transient behaviors of $w_1(t).$}
			\label{fig:Fig18a}
		\end{subfigure}
		\hfill
		\begin{subfigure}[b]{0.49\linewidth}
			\centering
			\includegraphics[width=\linewidth]{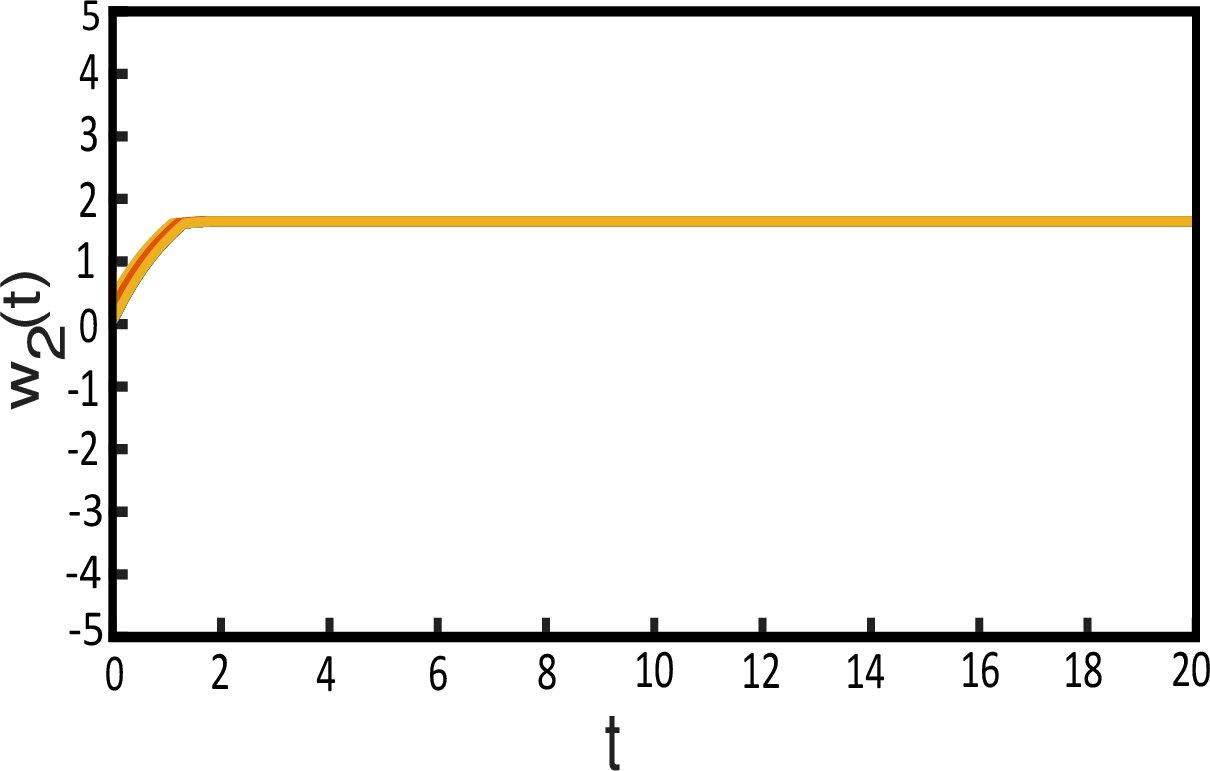}
			\captionsetup{labelformat=empty}
			\caption{Figure (12b): Transient behaviors of $w_2(t).$}
			\label{fig:Fig18b}
		\end{subfigure}
		\vspace{0.8em}
		\begin{subfigure}[b]{0.49\linewidth}
			\centering
			\includegraphics[width=\linewidth]{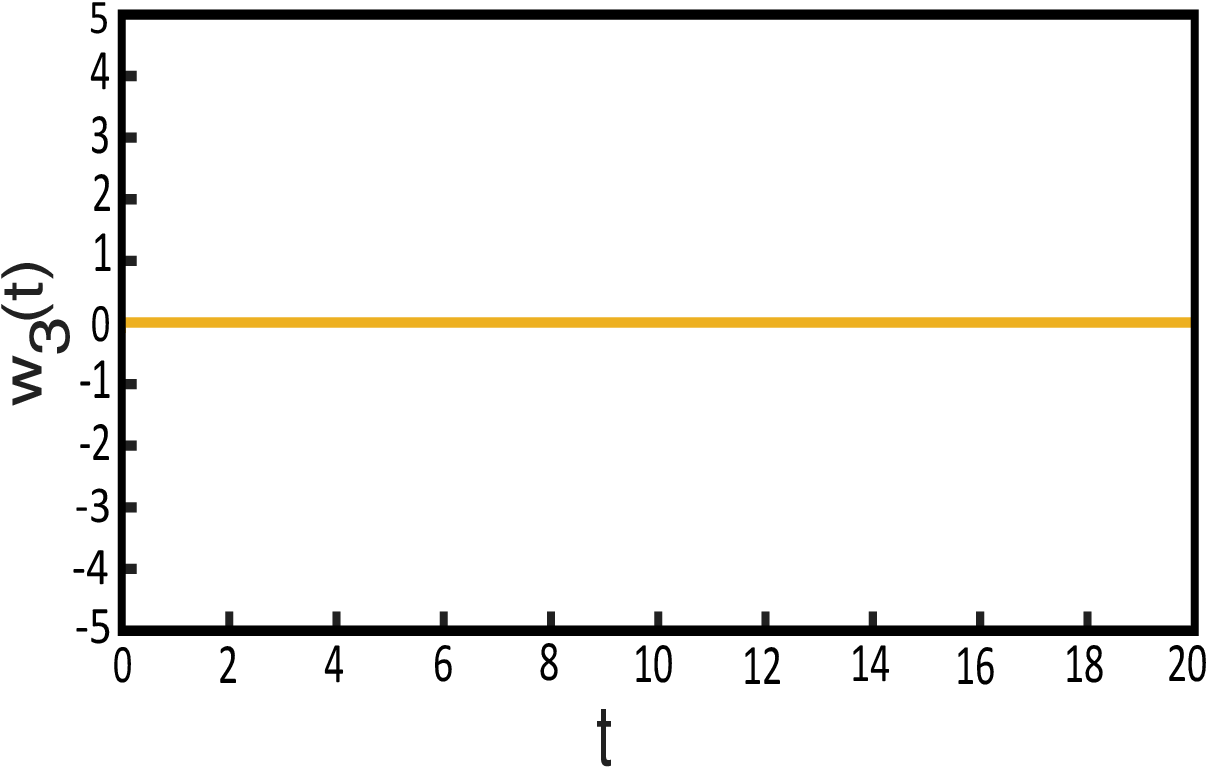}
			\captionsetup{labelformat=empty}
			\caption{Figure (12c): Transient behaviors of $w_3(t).$}
			\label{fig:Fig18c}
		\end{subfigure}
		\hfill
		\begin{subfigure}[b]{0.49\linewidth}
			\centering
			\includegraphics[width=\linewidth]{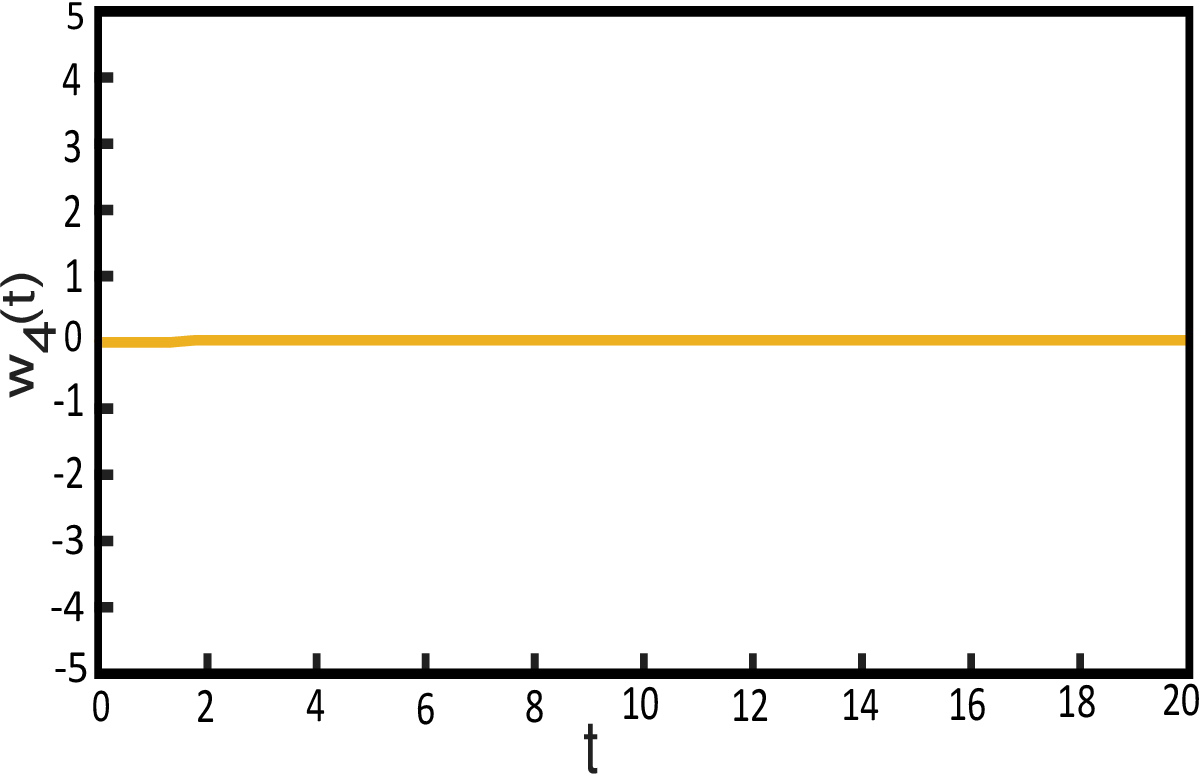}
			\captionsetup{labelformat=empty}
			\caption{Figure (12d): Transient behaviors of $w_4(t).$}
			\label{fig:Fig18d}
		\end{subfigure}
		\\
		\begin{subfigure}[b]{0.49\linewidth}
			\centering
			\includegraphics[width=\linewidth]{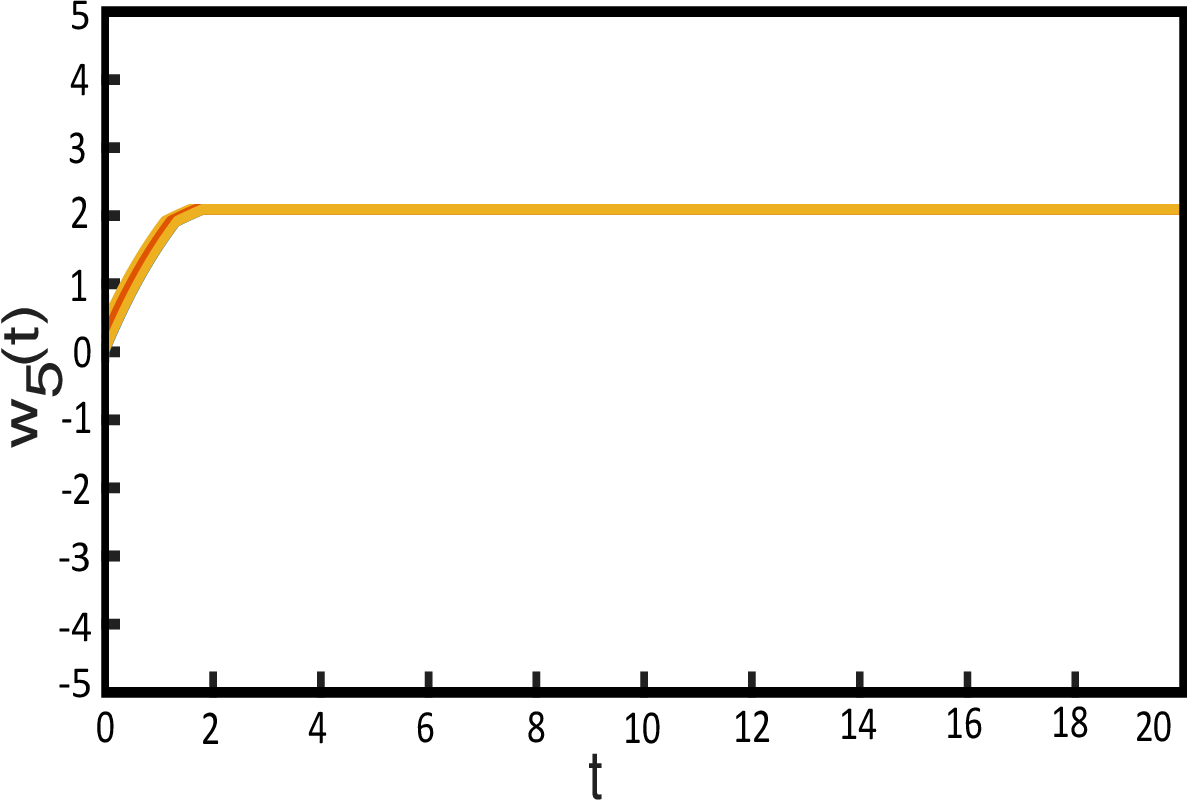}
			\captionsetup{labelformat=empty}
			\caption{Figure (12e): Transient behaviors of $w_5(t).$}
			\label{fig:Fig18e}
		\end{subfigure}
	\end{figure}
	
	\begin{table}[h!]
		\centering
		\setlength{\tabcolsep}{0.1pt}
		\renewcommand{\arraystretch}{1.2}
		\begin{tabular}{|c|c|c|c|}
			\hline
			Method &	Intial points &\makecell{Approximate optimal solution\\[-0.7ex] $(w_1^*, w_2^*,w_3^*, w_4^*,w_5^*)$}& \makecell{Approximate optimal value \\[-0.7ex] $f(w_1^*, w_2^*,w_3^*, w_4^*,w_5^*)~$} \\ 
			\hline
			\makecell{$l_1$\\[-0.7ex] penalty algorithm} 	& (0,0,0,0,0) &
			$  (2.5528,1.6409,0,0.0329,2.0921)$ &
		 2.8827
			\\
			\hline
			\makecell{$l_1$\\[-0.7ex] penalty algorithm} 	&	(0, 0, 0, 0, 10) &
			(2.5528, 1.6409, 0, 0.0329, 2.092) &
		 2.8827\\
			\hline
			\makecell{Lower-order\\[-0.7ex] penalty algorithm} 	& (0, 0, 0, 0, 0) &  (2.5528, 1.6409, 0, 0.0329, 2.092) &
			 2.8827\\
			\hline
			\makecell{Lower-order\\[-0.7ex] penalty algorithm} 	& (0, 0, 0, 0, 10) &
			(2.5528, 1.6409, 0, 0.0328, 2.092) &
			 2.8827
			\\
			\hline
			\makecell{Quadratic\\[-0.7ex] penalty algorithm} 	&	(0, 0, 0, 0, 0) &
			 (2.5528, 1.6409, 0, 0.0329, 2.092) &
			 2.8827\\
			\hline
			\makecell{Quadratic\\[-0.7ex] penalty algorithm} 	& (0, 0, 0, 0, 0) &  (2.5528, 1.64, 0.0001, 0.0329, 2.092) &
			2.8825
			\\
			\hline
			\makecell{Proposed\\[-0.7ex] neural network}	&	10 random &
			\makecell{(2.552739, 1.640964, -0.000039,\\ 0.032896, 2.092187)} &
		2.88254  \\
			\hline
		\end{tabular}
		\caption{}
		\label{tab:11}
	\end{table}
\end{Example}
\newpage
\section*{Observation:}
The proposed  neural network framework provides a  competitive way for solving the problem MPCC. Unlike the penalty-function algorithms ($l_1$, lower-order, and quadratic), which are implemented from a single prescribed starting point and require repeatedly solving a sequence of smooth penalty subproblems with updated parameters. The proposed  method generates solution trajectories directly from the continuous-time descent dynamics of an energy function and  use of  10 random initialization improves the solution reliability. Across the tested problems (Examples \ref{ex:5}, \ref{ex:6} and \ref{ex:7}), the approximate optimal solutions produced by the proposed neural network \eqref{eq:3} achieve optimal values that are consistently comparable to penalty algorithms and in many cases smaller than the values obtained by  Yang and Huang\cite{27} for the same  problems, while simultaneously maintaining the complementary structure through regularization. Also, for the problem MPCC6,  Outrata \cite{9} obtained the optimal value  as 
3.2077, while the proposed neural network  \eqref{eq:3} gives the optimal value   2.88254, which is smaller than those obtained by  Outrata \cite{9}. 
\par These observations indicate that the proposed neural network demonstrates improved numerical accuracy and enhanced solution reliability. Further, the proposed neural network is not only numerically stable and easy to implement (via standard ODE solvers), but also effective in producing high-accuracy  optimal solutions for the problem MPCC.
\section{Conclusion}
In this study, a gradient-based  neural network constructed to solve the mathematical programming problems  with complementary constraints MPCC. By employing the regularization techniques, the problem MPCC is transformed into an equivalent relaxed nonlinear programming problem NLP($\beta$). Using the penalty  function methods an energy function was constructed, whose minimum  represents the optimal solution to the  problem NLP($\beta$). The neural network has a simple one-layer architecture and low computational complexity which  makes it well-suited for hardware implementation and real-time applications. Based on Lyapunov stability theory and the LaSalle invariance principle,  the proposed neural network  asymptotic stable over its equilibrium point. Further,  computer simulations validate the proposed model is effective and reliable in obtaining the optimal solution of the problem MPCC. This approach can be  further extended to the multi-objective  and conic semi-infinite mathematical programming problems with  complementary constraints  incorporating the concepts of generalized convexity.\\
\par
\noindent
\textbf{Funding:} No funding was received to assist with the preparation of this manuscript.\\
\par
\noindent
\textbf{Author Contributions:} All authors contributed equally to this work.\\
\par
\noindent
\textbf{Data availability:} The data will be shared upon reasonable request..
\section*{Declarations} 
\noindent
\textbf{Conflict of interest: } The authors assert that they have no conflicts of interest.

\end{document}